\documentclass{amsart}

\usepackage[all]{xy}
\usepackage{amsmath}
\usepackage{hyperref}
\usepackage{amsfonts,graphics,amsthm,amsfonts,amscd,latexsym}
\usepackage{epsfig}
\usepackage{flafter}
\usepackage{mathtools}
\usepackage{comment}
\usepackage{stmaryrd}
\usepackage{tabularx}
\usepackage{pst-node}
\usepackage{tikz-cd}
\hypersetup{
    colorlinks=true,    
    linkcolor=blue,          
    citecolor=blue,      
    filecolor=blue,      
    urlcolor=blue           
}
\usepackage{enumitem}

\usepackage[normalem]{ulem}
\usepackage{pgfplots}
\pgfplotsset{compat=1.15}
\usepackage{mathrsfs}
\usepackage{tikz}
\usetikzlibrary{graphs,positioning,arrows,shapes.misc,decorations.pathmorphing}

\tikzset{
    >=stealth,
    every picture/.style={thick},
    graphs/every graph/.style={empty nodes},
}

\tikzstyle{vertex}=[
    draw,
    circle,
    fill=black,
    inner sep=1pt,
    minimum width=5pt,
]
\usepackage[position=top]{subfig}
\usepackage{amssymb}
\usepackage{color}

\calclayout

\usetikzlibrary{decorations.pathmorphing}
\tikzstyle{printersafe}=[decoration={snake,amplitude=0pt}]

\newcommand{\id}{\operatorname{id}}

\newcommand{\codim}{\operatorname{codim}}

\newcommand{\rank}{\operatorname{rank}}

\newcommand{\pp}{\mathbb{P}}

\newcommand{\qq}{\mathbb{Q}}

\usepackage{tikz}
\usetikzlibrary{matrix,arrows,decorations.pathmorphing}
\usetikzlibrary{arrows}

\definecolor{uuuuuu}{rgb}{0.26666666666666666,0.26666666666666666,0.26666666666666666}

  \newtheorem{theorem}{Theorem}[section]
  \newtheorem{lemma}[theorem]{Lemma}
  \newtheorem{proposition}[theorem]{Proposition}
  \newtheorem{corollary}[theorem]{Corollary}
  \newtheorem{claim}[theorem]{Claim}
  \newtheorem{conjecture}[theorem]{Conjecture}

  \newtheorem{notation}[theorem]{Notation}
  \newtheorem{definition}[theorem]{Definition}
  \newtheorem{example}[theorem]{Example}

  \newtheorem{question}[theorem]{Question}
  \newtheorem{construction}[theorem]{Construction}

\newtheorem{remark}[theorem]{Remark}

\theoremstyle{remark}

\numberwithin{equation}{section}



\begin{document}

\title[Hyperbolicity of adjoint linear series on varieties with positive tangent bundle]{Hyperbolicity of adjoint linear series on varieties\\ with positive tangent bundle}

\author[A.~Ito]{Atsushi Ito}
\address{Department of Mathematics, Institute of Pure and Applied Sciences, University of Tsukuba, Tsukuba, Ibaraki, 305-8571, Japan}
\email{ito-atsushi@math.tsukuba.ac.jp}

\author[J.~Moraga]{Joaqu\'in Moraga}
\address{UCLA Mathematics Department, Box 951555, Los Angeles, CA 90095-1555, USA
}
\email{jmoraga@math.ucla.edu}

\author[D.~Raychaudhury]{Debaditya Raychaudhury}
\address{Department of Mathematics, University of Arizona, Tucson, AZ 85721, USA}
\email{draychaudhury@arizona.edu}

\author[W. Yeong]{Wern Yeong}
\address{UCLA Mathematics Department, Box 951555, Los Angeles, CA 90095-1555, USA
}
\email{wyyeong@math.ucla.edu}

\thanks{The first author was partially supported by JSPS KAKENHI Grant Number 21K03201. The second author was partially supported by NSF research grant DMS-2443425. The third and the fourth authors were partially supported by AMS-Simons Travel Grants.}

\begin{abstract}
Let $X$ be a smooth projective variety of dimension $n\geq 3$, and let $L$ be an ample line bundle on $X$. In this article, we study the algebraic hyperbolicity of a very general section of the adjoint linear series $|K_X+mL|$ when the tangent bundle $\mathcal{T}_X$ of $X$ has suitable positivity properties. As a consequence, we show that the linear system $|K_X+mL|$ is hyperbolic (or pseudo-hyperbolic) for $m\geq 3n+1$, for various classes of polarized pairs $(X,L)$, thus providing new evidence of a conjecture that was proposed by the second and fourth authors.
Moreover, when $X$ is abelian, we show that the linear system $|mL|$ is hyperbolic for $m\geq n$, and the same holds when $m\geq n-1$, if $|L|$ has no base divisors. It turns out that these bounds for abelian varieties are sharp.
We also prove analogous statements for Kummer varieties and certain classes of hyperelliptic varieties. 
\end{abstract}

\maketitle

\vspace{-10pt}

\setcounter{tocdepth}{1}

\tableofcontents

\vspace{-30pt}

\section{Introduction}\label{intro}

Kobayashi hyperbolicity 
is a notion of hyperbolicity 
depending on a pseudo-metric associated to any
complex manifold (see~\cite{Kob76}).
Kobayashi hyperbolic manifolds
are important classes of complex manifolds
on which the Kobayashi pseudo-metric
is indeed a metric.
For instance, Kobayashi hyperbolic manifolds admit no non-constant maps from the complex line $\mathbb{C}$. There are several other related hyperbolicity notions such as Brody and arithmetic hyperbolicity, see \cite{Jav} for a survey.
In \cite{Dem97}, Demailly introduced the notion of algebraic hyperbolicity for a smooth projective variety $X$ as
an algebraic analogue of Kobayashi hyperbolicity (see Definition~\ref{def:alg-hyp}). The algebraic hyperbolicity condition
imposes a uniform lower bound
on the genus of curves $C\hookrightarrow X$
which is linear in terms of the degree of $C$ 
with respect to some fixed polarization on $X$. It is expected that (very) general hypersurfaces of large degree on smooth projective varieties are hyperbolic.
There are many theorems incarnating this principle.
For instance, in~\cite{MR3735863}, Brotbek proved that for any smooth projective variety $X$ and an ample line bundle $L$ on $X$, there exists a constant $d_0:=d_0(X,L)$ such that 
a general element $H\in |d_0L|$ is Kobayashi hyperbolic. 
For a very ample line bundle $L$ we know that it suffices to take $d_0\leq (n+1)^{2n+6}$ where $n$ is the dimension of $X$. 
In the case of algebraic hyperbolicity, there are many effective results for projective space.
For instance, very general hypersurfaces in $\pp^n$ of degree $d$ are algebraically hyperbolic if $d\geq 2n-2$ and $n\geq 5$, or if $d\geq 2n-1$ and $n\geq 3$ \cite{CL1, CL2, CoRi1, Ein88, Pac04, Voi96, Voi98, Yeo24}. 
It follows from the argument in \cite{Ein88} that for any smooth projective variety $X$ of dimension $n$ and any ample and globally generated line bundle $L$ on $X$, the linear system $\left|dL \right|$ is {algebraically} hyperbolic whenever $d\geq 2n$.

\smallskip

Motivated by Fujita-type conjectures, the second and fourth authors have recently made the following

\begin{conjecture}[{\cite[Conjecture 1.1]{MY24}}]\label{conj}
    Let $X$ be a smooth projective variety of dimension $n\geq 3$, and let $L$ be an ample line bundle on $X$. Then the linear system $|K_X+mL|$ is hyperbolic if $m\geq 3n+1$.
\end{conjecture}

This conjecture was confirmed in~\cite[Theorem 1.2]{MY24} for smooth toric varieties
as well as in some cases of Gorenstein toric varieties (see, e.g.,~\cite[Theorem 1.3]{MY24}).
More recently, it was studied for general Fano threefolds with Picard number one \cite{Seo25}. 
Moreover, the conjecture has been confirmed for
smooth projective spherical varieties with smooth orbit closures \cite{KwSe25}, which includes rational homogeneous varieties and toric varieties (generalizing \cite[Theorem 1.2]{MY24}).
In this article, we study the hyperbolicity
of adjoint linear systems on projective varieties whose tangent bundles are suitably positive. 

\smallskip

We phrase our results in terms of nef values of polarized pairs whose definition we now recall. Given a polarized smooth projective variety $(X,{L})$ of dimension $n$, its {\it nef value} is defined as follows $$\tau({L}):=\min\left\{t\in\mathbb{R}\mid K_X+tL\textrm{ is nef}\right\}.$$ It is known that $\tau(L)\leq n+1$, and it is $\leq n$ unless $(X,L)\cong(\mathbb{P}^n,\mathcal{O}_{\mathbb{P}^n}(1))$ (Maeda). We refer to \cite{BS} for a comprehensive literature regarding the structures of polarized algebraic varieties with large nef values.

\smallskip

Our first result is for regular varieties with very ample polarizations:

\begin{theorem}\label{main1}
Let $X$ be a smooth regular projective variety of dimension $n\geq 3$ and let
${L}$ be a very ample line bundle on $X$.
If the tangent bundle $\mathcal{T}_X$ is nef (resp. pseudo-nef, almost nef), 
then the linear system $|K_X+mL|$ is hyperbolic (resp. pseudo-hyperbolic, almost hyperbolic) for $m>\max\left\{n,n+2\tau({L})-2\right\}$.
\end{theorem}

We refer to Definition \ref{def:alg-hyp-ls} for the various variants of the hyperbolicity notion used in the above statement. Observe that the number $n+2\tau(L)-2\leq 3n$, whence the above proves Conjecture \ref{conj} for any regular variety with nef tangent bundle  for very ample polarizations. 

\smallskip



It is known that projective (almost) homogeneous spaces have (almost) nef tangent bundle (see Section \ref{prelim} for the definitions).
We have an ``almost homogeneous variant'' of the above, see Corollary \ref{cor-ahvariant}.

\smallskip

We now turn our attention to irregular varieties with nef tangent bundle. We remark that algebraic hyperbolicity of irregular varieties is not very well-studied. One of the well-known results in this direction is Bloch's conjecture \cite{Blo26}, which asserts that if $q(X)>\dim X$, then every entire holomorphic curve $f:\mathbb{C}\to X$ is algebraically degenerate, i.e., $f(\mathbb{C})$ lies in a proper subvariety of $X$.
This result was proved by Kawamata \cite{Kaw80}, who verified Ochiai's conjectured lemma that is sufficient for Bloch's conjecture (see \cite{Och77}).

\smallskip

When $X$ is a complex torus, Green \cite{Gre78} proved that a closed complex subspace of $X$ is Kobayashi hyperbolic if and only if it contains no translate of a positive-dimensional subtorus of $X$.
In fact, for subvarieties of abelian varieties, we note the following that has been observed by Caucci \cite[Remark 2.2]{caucci_KH} after the first version of this work was posted on the arXiv:

\begin{remark}\label{Bloch}
    A subvariety of an abelian variety is algebraically hyperbolic if and only if it contains no translate of a positive-dimensional abelian subvariety.
\end{remark}

Using \cite{Br, Dem97, Gre78, JK}, he also observes that for subvarieties of abelian varieties, the notions of Kobayashi, Brody and algebraic hyperbolicity are equivalent. It is worth mentioning that by the work of Yamanoi, a general type subvariety $X$ of an abelian variety is algebraically hyperbolic outside the special subset of $X$\footnote{Here, the special subset of $X$ is the Zariski closure of the union of all images of non-constant rational maps $A\dashrightarrow X$ from abelian varieties $A$.} which is proper (see, e.g.,~\cite[Corollary 1]{Yam15}).

As mentioned, we mostly deal with the case when the tangent bundle $\mathcal{T}_X$ of a smooth projective variety $X$ of dimension $\geq 3$ is nef, in which case $q(D)\leq \dim D$ for any ample divisor $D$ of $X$. It is a result of Campana--Peternell \cite{CP91} that the Albanese morphism of such varieties are smooth fibrations. Moreover, by a result of Demailly--Peternell--Schneider \cite{DPS}, the Albanese map, up to an \'etale cover, is a smooth fibration whose fibers are Fano varieties with nef tangent bundle \cite[Main Theorem]{DPS} (which are conjecturally even rational homogeneous spaces by Campana--Peternell Conjecture). For such fibrations, we have 

\begin{theorem}\label{main4}
    Let $X$ be a smooth projective irregular variety of dimension $n\geq 3$ with nef tangent bundle $\mathcal{T}_X$. 
    Assume that the Albanese morphism ${\rm alb}_X\colon X \rightarrow A$ of $X$ has regular fibers. Let $L$ be an ample and globally generated line bundle on $X$, and $m$ be an integer  such that 
    \begin{enumerate}
        \item[\emph{(i)}] $L|_{F_p}$ is very ample for all fibers $F_p$ of ${\rm alb}_X$, and 
        \item[\emph{(ii)}] $m>\max\left\{n,n+2\tau({L})-2\right\}$ is an integer such that
        $mL-\beta(L_A){\rm alb}_X^*(L_A)$ is ample for some ample line bundle $L_A$ on $A$.
    \end{enumerate}  Then the linear system $|K_X+mL|$ is hyperbolic. 
\end{theorem}

In the above, $\beta(-)$ is the base-point freeness threshold introduced by Jiang--Pareschi \cite{JP}: given a polarized abelian variety $(A,L)$, the invariant $\beta(L)\in (0,1]$ measures the positivity of $L$, in particular, $L$ is base-point free if and only if this invariant is $<1$ (see Section \ref{positivity} for details).

\smallskip

Our next result is an immediate consequence of Theorem \ref{main4}:

\begin{corollary}\label{cor4}
{Let $X$ be a smooth projective irregular variety of dimension $n\geq 3$ with nef tangent bundle $\mathcal{T}_X$. Assume that the Albanese morphism ${\rm alb}_X\colon X \rightarrow A$ of $X$ has regular fibers. 
Let $L$ be an ample and globally generated line bundle on $X$ such that
\begin{itemize}
\item[\em{(i)}] $L|_{F_p}$ is very ample for every fiber $F_p$ of ${\rm alb}_X$ (which holds, for example, if $L$ is very ample, or if the fibers are rational homogeneous spaces), and
\item[\em{(ii)}]  ${L}- \frac{1}{n+1}\cdot \mathrm{alb}_X^*({L}_A) $ is ample for some ample line bundle $L_A$ on $A$.
\end{itemize}
Then the linear system $|K_X+mL|$ is hyperbolic for $m>\max\left\{n,n+2\tau({L})-2\right\}$.}
\end{corollary} 

\begin{remark}
\em{Recall that the {\it stable irregularity} $\Tilde{q}(X)$ is defined as follows:
$$\Tilde{q}(X):=\max\left\{q(\Tilde{X})\mid \text{there exists } f:\Tilde{X}\to X\textrm{ where $f$ is \'etale, and $\Tilde{X}$ is compact and connected}\right\}.$$
By \cite[Proposition 2.8]{CP91}, if a smooth projective irregular variety $X$ with nef tangent bundle satisfies $\Tilde{q}(X)=q(X)$, then $q(F_p)=0$ for every fiber $F_p$ of its Albanese morphism.}
\end{remark}

The above corollary is applicable in many situations; we give one example below:

\begin{example}
\em{Let $\mathcal{E}$ be a semistable bundle of rank $r\geq 3$ and slope $\mu$ on an elliptic curve $E$, and set $\pi:X:=\mathbb{P}(\mathcal{E})\to E$. It follows that $|K_X+mL|$ is a hyperbolic linear system for all $m\geq \dim X+1=r+1$ if $L$ is ample and base-point free. 

Indeed, since the relative tangent bundle $\mathcal{T}_{X/E}$ of $\pi$ is nef (see for e.g., \cite[Proof of Proposition 6.1]{Lee}), it follows that $\mathcal{T}_X$ is nef by the definition of $\mathcal{T}_{X/E}$.} Denote the numerical class of a fiber by $F$ and let $\lambda_{\mathcal{E}}:=c_1(\mathcal{O}_{\mathbb{P}(\mathcal{E})}(1))-\mu F$. By \cite[Theorem 3.1]{Miy}, we have $${\rm Nef}(X)=\mathbb{R}_{\geq 0}\lambda_{\mathcal{E}}\oplus\mathbb{R}_{\geq 0}F.$$
If $L\equiv a\lambda_{\mathcal{E}}+bF$ is ample and globally generated line bundle, then $a,b>0$. We claim that $b\geq \frac{1}{r}$. To see this, note that it must be the case that $L \equiv s c_1(\mathcal{O}_{\mathbb{P}(\mathcal{E})}(1)) +t F$ for $s,t \in \mathbb{Z}$.
Then $L \equiv s\lambda_{\mathcal{E}} + (s \mu +t) F $ and hence $b= s \mu +t >0$ implies $b \geq \frac{1}{r}$. Thus the consequence follows from Corollary \ref{cor4}.
\end{example}

Among the classes of varieties with nef tangent bundles, some special ones are the classes of abelian varieties, or more generally homogeneous spaces, and finite \'etale quotients of abelian varieties (i.e., hyperelliptic varieties). Let us start by discussing the hyperbolicity property of adjoint linear systems on abelian varieties. 

\smallskip

It is now well-established that sheaves on abelian varieties enjoy remarkable regularity properties. The study on this topic dates back to Green--Lazarsfeld's generic vanishing condition \cite{GL}, which was subsequently studied by Hacon \cite{Hac} and extended vastly by Pareschi--Popa \cite{PaPo1, PP2, PaPo3}. The notion of M-regularity emerged through these works, which was extended to $\mathbb{Q}$-twisted setting by Jiang--Pareschi \cite{JP}. Our next result is a criterion for hyperbolicity of the linear system $|L|$ through its jet-separation property and M-regularity (see Section \ref{prelim} for definitions):

\begin{theorem}\label{mainmod}
    Let $(X,L)$ be a polarized abelian variety of dimension $g\geq 3$ and $0 \leq k \leq g-1$. Assume that $X$ contains no positive-dimensional abelian subvarieties of dimension at most $k$. \footnote{Note that this condition is vacuously true if $k=0$.}
Let $p \geq 1$ be an integer such that
\begin{align}\label{eq_cor_p-jet_ample}
\binom{p+k+1}{k+1} + k+1 \geq g.
\end{align}
If $L$ separates $p$-jets at any $x \in X$,
then the linear system $|L|$ is hyperbolic.
In particular, $|L|$ is hyperbolic
if $\mathcal{I}_o \langle \frac{1}{p+1} L \rangle $ is M-regular.
\end{theorem}

\begin{remark}\label{rmk1}
\em{Taking $k=0$ and $p=g-2$ in the above, we see that if $L$ separates $(g-2)$-jets at any $x \in X$, then the linear system $|L|$ is hyperbolic.
In particular,  $|L|$ is hyperbolic if 
$\mathcal{I}_o \langle \frac{1}{g-1} L \rangle $ is M-regular. 
This was originally proven in the first version of this manuscript by using Theorem \ref{thm:hyperbolicity-of-N+L} below. (See Remark \ref{rem_proof_of_first_version}.)}
\end{remark}


As an immediate consequence, we obtain the following

\begin{corollary}\label{cor_KH_mL}
Let $(X,L)$ be a polarized abelian variety of dimension $g\geq 3$ and $0 \leq k \leq g-1$. Assume that $X$ contains no positive-dimensional abelian subvarieties of dimension at most $k$.
Then: 
\begin{enumerate}
\item[$(1)$] The linear system $|mL| $ is hyperbolic if $m \geq 3$ and  $\binom{m+k-1}{k+1} + k+1 \geq g$.
\item[$(2)$] If $L$ has no base divisor, $|mL| $ is hyperbolic if $m\geq 2$ and $\binom{m+k}{k+1} + k+1 \geq g$.
\end{enumerate}
\end{corollary}

To clarify the meaning of this result, let us state an immediate corollary that is an extension of the following remark.

\begin{remark}\em{
If $L$ is an ample and base-point free line bundle on a simple abelian variety $X$ of dimension $g\geq 3$, then the linear system $|L|$ is hyperbolic. Indeed, as $X$ is simple, it does not contain any abelian proper subvarieties, therefore Remark \ref{Bloch} implies that $|L|$ is hyperbolic.}
\end{remark}

\begin{corollary}\label{cor_g-k}
Let $(X,L)$ be a polarized abelian variety of dimension $g\geq 3$ and $0 \leq k \leq g-1$. 
Assume that $X$ contains no positive-dimensional abelian subvarieties of dimension at most $k$.
Then 
\begin{enumerate}
\item[$(1)$] The linear system $|mL| $ is hyperbolic  if $mL$ is base-point free and $m \geq g-k$.
\item[$(2)$] If $L$ has no base divisor, $|mL| $ is hyperbolic if $mL$ is base-point free and $m \geq \max\left\{1,g-k-1\right\}$.
\end{enumerate}
\end{corollary}

\begin{remark}
\em{Since $\binom{m+k-1}{k+1} \geq \frac{(m-1)^{k+1}}{(k+1)!}$,
the condition in Corollary \ref{cor_KH_mL}(1) is satisfied if $$ m \geq ((k+1)!(g-k-1))^{\frac1{k+1}} +1.$$
Similarly, the condition in (2) is satisfied if $$ m \geq  ((k+1)!(g-k-1))^{\frac1{k+1}}.$$ 
Hence Corollary \ref{cor_KH_mL} is much stronger than Corollary \ref{cor_g-k} if $g \gg k \geq 1$.
For example, for $k=1$, the condition in Corollary \ref{cor_KH_mL}(1) is $ m \geq \frac{1+\sqrt{8g-15}}{2}$, which is weaker than $m \geq g-1$ if $g \geq 5$.}
\end{remark}

It is natural to wonder how optimal the above results are for abelian varieties. In fact, recall that Conjecture \ref{conj} is motivated by Fujita-type conjectures, and for polarized abelian varieties $(X,L)$, we have base-point freeness (resp.\ the very ampleness) of $|mL|$ as soon as $m\geq 2$ (resp.\ $m \geq 3$) by Lefschetz theorem. 
Also observe the following:

In view of the above discussions, for polarized abelian varieties $(X,L)$, it is quite tempting to wonder if there is a constant lower bound on $m$ for which the linear system $|mL|$ will satisfy hyperbolicity in {\it every} dimension.  
However, this expectation turns out to be false:

\begin{proposition}\label{intro-prop1-Ab}
Let $(X',L')$ be a polarized abelian variety of dimension $g- 1\geq 2$ and let $(E,\Theta)$ be a principally polarized elliptic curve. If
\begin{equation}\label{str}
    (X,L)\cong (X',L')\times (E,\Theta).
\end{equation}
then the linear system $|(g-1){L}|$ is not hyperbolic.
\end{proposition}

Note that Lazarsfeld's conjecture, proven by Pareschi \cite{Par} shows us that $(g-1)L$ is projectively normal if $g \geq 4$, and satisfies $(N_p)$ if $g \geq p+4$ (see e.g., {\it loc. cit.} for the definitions).
Hence $(N_p)$ does not imply algebraic hyperbolicity even if $p$ is sufficiently large unless we fix the dimension. 

\smallskip

After the first version of this article was posted, Caucci discovered that if $(X,L)$ is a polarized abelian variety of dimension $g\geq 3$, then $|(g-1)L|$ is hyperbolic unless $(X,L)$ is as in \eqref{str} \cite[Theorem A]{caucci_KH}. His proof combines results of \cite{Pa24,AP} with Theorem \ref{thm:hyperbolicity-of-N+L} below. We give an alternative proof of this fact in Theorem \ref{caa}.


\smallskip

Theorem \ref{mainmod} also gives us 
corollaries about hyperbolicity 
on {\it general} polarized abelian varieties:

\begin{corollary}\label{cor-ab-2}
Let $(X,{L})$ be a general polarized abelian variety of dimension $g\geq 3$ and of type $(d_1,\dots,d_g)$. 
If 
\[
d_1\cdots d_g > \frac{4^g(g-1)^gg^g}{2g!} 
\]
then $|{L}|$ is a hyperbolic linear system.
\end{corollary}

We note that 
the above result has been refined by Caucci using Remark \ref{rmk1} and some recent literature \cite[Proposition E and Proposition 5.5]{caucci_KH}.

We have a better bound when $d_1=d_2=\cdots=d_{g-1}=1$:

\begin{corollary}\label{cor-ab-3}
Let $(X,{L})$ be a general polarized abelian variety of dimension $g\geq 3$ and of type $(1,1,\dots,1,d)$ with 
\[
d\geq (g-1)^g + \dots + (g-1)+1.
\]
Then $|{L}|$ is a hyperbolic linear system.
\end{corollary}



We now turn to the case of projective homogeneous spaces that are not rational homogeneous or abelian. 

\begin{theorem}\label{thm3'}
Let $(X,L)$ be a polarized projective homogeneous space of dimension $n\geq 3$ and irregularity $1\leq q(X)\leq n-1$. Then the following statements hold:
\begin{itemize}
    \item[$(1)$] The linear system $|K_X+mL|$ is hyperbolic if $m\geq n+2\tau(L)-1$.
    \item[$(2)$] If $L$ is base-point free, then $|K_X+mL|$ is hyperbolic if $m> \max\left\{n-q(X),n+2\tau(L)-2\right\}$. 
\end{itemize}
\end{theorem}

\begin{remark}
\em{Our results completely prove Conjecture \ref{conj} for any polarized projective homogeneous space $(X,L)$ of dimension $n\geq 3$ (with much better bounds unless $(X,L)\cong(\mathbb{P}^n,\mathcal{O}_{\mathbb{P}^n}(1))$).

Indeed, as mentioned before, these varieties have nef (in fact globally generated) tangent bundle. If $q(X)=0$, then they are rational homogeneous spaces, in which case any ample line bundle is very ample, whence the conclusion follows from Theorem \ref{main1}. If $q(X)=n$ (resp. $1\leq q(X)\leq n-1$), then the conclusion follows from Corollary \ref{cor_g-k} (resp. Theorem \ref{thm3'}).}
\end{remark}

In the case of finite {\'etale} quotients of abelian varieties, using the results of \cite{Cau25} we obtain:

\begin{theorem}\label{hyp}
Let $X=A/G$ be a hyperelliptic variety obtain from a free action of a commutative group $G$ on an abelian variety $A$. Assume $n:=\dim X\geq 3$. Let $L$ be an ample line bundle and  $N\equiv 0$ be a numerically trivial line bundle on $X$. Then
\begin{itemize}
    \item[$(1)$] The linear system $|{N}+mL|$ is hyperbolic for $m\geq 2n+1$.
    \item[$(2)$] If $q(X)>0$, then $|{N}+mL|$ is hyperbolic for $m\geq 2n$.
\end{itemize}
\end{theorem}

As explained in \cite{Cau25}, there are many varieties that satisfy the assumptions above, and Bagnera-de Franchis varieties are one such. We include an immediate consequence for them in Corollary \ref{BdF}.

\smallskip

The key technical ingredient in our proofs is a general criterion of (pseudo-, almost) hyperbolicity of a base-point free linear system when the tangent sheaf of the variety is suitably positive (see Theorem \ref{thm:hyperbolicity-of-N+L} and Theorem \ref{ahvariant}). These results allow us to study hyperbolicity by means of the positivity properties of $\mathbb{Q}$-twisted syzygy bundles. Because of this, in case of irregular varieties, we can make effective use of the regularity properties of sheaves on abelian varieties. Moreover, our techniques also yield results for varieties with Gorenstein canonical singularities. To highlight this, we include our result for Kummer varieties, which are by definition quotients of abelian varieties by the inverse involution $a\mapsto -a$:

\begin{theorem}\label{thm-kummer}
    Let $(K(A),L)$ be a polarized Kummer variety associated to an abelian variety $A$ of dimension $n\geq 3$. Then the linear system $|mL|$ is hyperbolic for $m\geq \frac{n}{2}$.
\end{theorem}

Our techniques can be used to prove hyperbolicity results for various other classes of varieties. We work over the field $\mathbb{C}$ of complex numbers. A {\it variety} is an integral separated scheme of finite type over $\mathbb{C}$.

\smallskip

\noindent{\bf Organization.} We provide the basic definitions of hyperbolicity conditions and positivity properties in Section \ref{prelim}. 
Section \ref{positivity} is devoted to the study of positivity properties of twisted syzygy bundles. We explain how the positivity of syzygy bundles is related to hyperbolicity in Section \ref{transition}. We prove our main results in Section \ref{proofs}.

\smallskip

\noindent{\bf Acknowledgments.} We thank Federico Caucci, Mihnea Popa, Burt Totaro, and Katsutoshi Yamanoi for helpful communications. We also thank Haesong Seo and Guolei Zhong for pointing out an inaccuracy in the first version of this manuscript.

\section{Preliminaries on hyperbolicity and positivity}\label{prelim}
 
In this section, we recall some preliminary definitions regarding algebraic hyperbolicity, positivity properties of sheaves, and their $\mathbb{Q}$-twisted variants.

\subsection{Algebraic hyperbolicity} We start with the definition of (Demailly) algebraic hyperbolicity.

\begin{definition}\label{def:alg-hyp}
{\em 
We say that a projective variety $X$ is \emph{algebraically hyperbolic} if there is some $\varepsilon>0$ and ample line bundle $L$ such that every non-constant map $f: C \to X$ from a smooth projective curve $C$ of genus $g(C)$ satisfies the inequality
\begin{equation}
\label{eqn:hyperbolicity-definition}
2g(C)-2\geq \varepsilon \deg f^*L.
\end{equation}
}
\end{definition}

We remark that the algebraic hyperbolicity condition does not dependent on the choice of polarization, i.e., it holds true for $L$ if and only if it holds true for any ample line bundle. 
Indeed, if $H$ is another ample line bundle on $X$, then there exists $\delta>0$ such that $L-\delta H$ is ample. So the definition works for the pair $(H,\varepsilon\delta)$ provided that it works for $(L,\varepsilon)$.

\smallskip

It is reasonable to define hyperbolicity of the linear series associated to an ample and base-point free line bundle through that of a very general section of the series:

\begin{definition}\label{def:alg-hyp-ls}
{\em 
Let $E$ be an ample and base-point free  line bundle on a polarized projective variety $(X,L)$.
The line bundle $E$ (or the linear system $\left|E\right|$) is called: 
\begin{itemize}[leftmargin=0.8cm]
    \item \textit{Hyperbolic} if the variety $D_E$ in $X$ defined by the vanishing of a very general element of $\left|E\right|$ is algebraically hyperbolic.
    \item \textit{Pseudo-hyperbolic} if there is a proper closed subset $Z\subsetneq X$ and a constant $\varepsilon>0$ such that a very general element $D_E$ of $\left|E\right|$ is algebraically hyperbolic outside of $Z$, namely every non-constant map $f: C \to D_E$ from a smooth projective curve $C$ where $f(C)\not\subset Z$ satisfies \eqref{eqn:hyperbolicity-definition}.
    \item {\it Almost hyperbolic} if there is a countable family of subvarieties $Z_i\subsetneq X$ ($i\in\mathbb{N}$) and a constant $\varepsilon>0$ such that a very general element $D_E$ of $\left|E\right|$ is algebraically hyperbolic outside of $\bigcup_{i\in\mathbb{N}} Z_i$, i.e., every non-constant map $f: C \to D_E$ from a smooth projective curve $C$ where $f(C)\not\subset \bigcup_{i\in\mathbb{N}} Z_i$ satisfies \eqref{eqn:hyperbolicity-definition}.
\end{itemize}
}
\end{definition}

It turns out that the above definitions do not depend on the polarization $L$, i.e., one can define hyperbolicity of a line bundle and its variants without specifying any polarization. Moreover, one can also define {\it algebraic pseudo-} (resp. {\it almost}) {\it hyperbolicity} of a projective variety $X$ in a similar fashion. Note that $X$ (or a line bundle $E$ on $X$) is
\begin{center}
    hyperbolic$\implies$pseudo-hyperbolic$\implies$almost hyperbolic.
\end{center}

\begin{example}[Pseudo-hyperbolic but not hyperbolic variety] \em{Note that the converse to the first implication above does not hold. Indeed, take any smooth projective variety that is algebraically hyperbolic. If we blow-up a point, the resulting variety is algebraically pseudo-hyperbolic (outside the exceptional divisor) but not algebraically hyperbolic.}
\end{example}

However, we do not know the answer to the following 
\begin{question}
{\em Is there a smooth projective variety that is almost hyperbolic but not pseudo-hyperbolic?}
\end{question}

\subsection{Positivity of ($\mathbb{Q}$-twisted) sheaves and geometry} We recall the definitions of the positivity properties that will be useful. Given a sheaf $\mathcal{E}$ on a projective variety $X$, we denote $\mathrm{Proj}({\rm Sym}^{\bullet}(\mathcal{E}))$ by $\mathbb{P}(\mathcal{E})$.
\begin{definition}\label{def:gen-nef}
{\em We say that a coherent sheaf $E$ on $X$ is \emph{nef} (resp. \emph{ample}) if the line bundle $\mathcal{O}_{\pp(E)}(1)$ on $\pp(E)$ is nef (resp. ample). }
\end{definition}

It is known that $E$ is nef if and only if the line bundle $\mathcal{O}_{\pp(E\vert_C)}(1)$ on $\pp(E\vert_C)$ is nef for every integral curve $C\subset X$. Thus, the above definition  of nefness naturally extends to the following

\begin{definition}\label{pnef}
{\em A coherent sheaf $E$ on $X$ is 
\begin{itemize}[leftmargin=0.8cm]
    \item \emph{pseudo-nef} if there is a closed subset $Z\subsetneq X$ such that $\mathcal{O}_{\pp(E\vert_C)}(1)$ on $\pp(E\vert_C)$ is nef for every integral curve $C\not\subset Z$, 
    \item \emph{almost nef} if there is a countable family of subvarieties $Z_i\subsetneq X$ ($i\in\mathbb{N}$) such that $\mathcal{O}_{\pp(E\vert_C)}(1)$ on $\pp(E\vert_C)$ is nef for every integral curve $C\not\subset \bigcup_{i\in\mathbb{N}}Z_i$.
\end{itemize}}
\end{definition}

Clearly, for a sheaf $E$ on $X$, we have 
\begin{center}
    ample$\implies$nef$\implies$pseudo-nef$\implies$almost nef.
\end{center}

Recall that a variety $X$ is called {\it homogeneous} (resp. {\it almost homogeneous}) if ${\rm Aut}^0(X)$ acts transitively (resp. with an open orbit) on $X$. It is well-known that if a smooth projective variety is homogeneous (resp. almost homogeneous), then its tangent bundle $\mathcal{T}_X$ is nef (resp. almost nef). We also have the following structure theorem for varieties with nef tangent bundle:
\begin{theorem}[{\cite[Proposition 2.4]{CP91}}]\label{structure}
    Let $X$ be a smooth projective variety with nef tangent bundle $\mathcal{T}_X$. Then its Albanese map ${\rm alb}_X:X\to A$ is a surjective and smooth morphism whose fibers are connected with nef tangent bundles.
\end{theorem}

We also note the following where for a singular variety, we define its tangent sheaf as the reflexive hull of the push-forward of the tangent sheaf from its smooth locus. The following will allow us to prove a hyperbolicity statement on Kummer varieties.

\begin{proposition}\label{kummer}
Let $X$ be a Kummer variety. 
Let $C$ be a curve in $X$ that does not contain a singular point of $X$. Let $\mathcal{T}_X\rightarrow Q\rightarrow 0$ be a quotient coherent sheaf. Then, we have $\deg(Q|_C)\geq 0$. 
\end{proposition}

\begin{proof}
Let $A^0$ be the locus of the abelian variety where the involution is acting freely. Let $f\colon A^0\rightarrow X^0$ be the corresponding finite \'etale morphism. Note that $C\subset X^0$. We have 
$f^*\mathcal{T}_{X^0}= \mathcal{T}_{A^0}$. Thus, we conclude that $\mathcal{O}(1)$ is nef on $\mathbb{P}(\mathcal{T}_{X^0}|_C)$ and so $\mathcal{T}_X|_C$ is a nef vector bundle. 
Therefore, $Q|_C$ is a nef sheaf as well, so 
$\deg(Q|_C)\geq 0$. 
\end{proof}

We will frequently use the theory of $\mathbb{Q}$-twisted sheaves, and their positivity properties. 
Let $ L'$ be an arbitrary line bundle on $X$. Following \cite[Section 6.2A]{Laz'}, given a coherent sheaf $\mathcal{F}$ and a rational number $x$, define the $\mathbb{Q}$-twisted sheaf
$\mathcal{F}\langle x L'\rangle$ as follows: it is the equivalence class of pairs $(\mathcal{F},x L')$ where the equivalence is given by $$(\mathcal{F}\otimes  L'^{\otimes m},y L')\sim (\mathcal{F},(y+m) L').$$
The positivity properties of a $\mathbb{Q}$-twisted sheaf $\mathcal{F}\langle x L'\rangle$ can be defined through that of the $\mathbb{Q}$-twisted bundle  $\mathcal{O}_{\mathbb{P}(\mathcal{F})}(1)\langle x\cdot\pi^* L'\rangle$ as explained in {\it loc. cit.} where $\pi:\mathbb{P}(\mathcal{F})\to X$ is the structure map.

\subsection{Jet-separation and jet-ampleness} Finally, let us recall the definition of {\it jet-separation}. A line bundle $L$ on a projective variety $X$ is said to {\it separate $p$-jets at $x\in X$} if the map $$H^0(L)\to H^0(L\otimes\mathcal{O}_X/\mathcal{I}_x^{p+1})$$ is surjective, where $\mathcal{I}_x\subset\mathcal{O}_X$ is the ideal sheaf of $x\in X$. A related notion is jet-ampleness:  $L$ is said to be {\it $p$-jet ample} if 
$$H^0(L)\to H^0(L\otimes\mathcal{O}_X/\mathcal{I}_{x_1}\mathcal{I}_{x_2} \cdots \mathcal{I}_{x_{p+1}})$$ is surjective for any (not necessarily distinct) $x_1,\dots x_{p+1} \in X$. In particular, $L$ is 
\begin{itemize}
    \item $0$-jet ample if and only if $L$ is base-point free,
    \item $1$-jet ample if and only if $L$ is very ample.
\end{itemize}
Hence $p$-jet ampleness is stronger than the condition that $L$ separates $p$-jets at any $x\in X$ for $p \geq 1$. 

We remark that jet-ampleness of linear series on abelian varieties is related to M-regularity, see Theorem \ref{It} for details.

\section{Positivity of syzygy bundles}\label{positivity}

This section is devoted to the study of positivity properties, such as global generation, or more generally nefness of twisted syzygy bundles whose definition we first recall. 

Let $X$ be a smooth projective variety and let $ L$ be a globally generated line bundle on $X$. The {\it syzygy bundle} $M_{ L}$ of $ L$ is by definition  $$M_{ L}:=\mathrm{Ker}(H^0( L)\otimes \mathcal{O}_X\xrightarrow{\textrm{ev}}  L)$$ 
where the map above is the evaluation map of global sections of $ L$. 

\subsection{Preliminaries on syzygy bundles}
We start by providing  some preliminaries on syzygy bundles. The following proposition leads to a corollary which will be used throughout:
\begin{lemma}\label{l1}
    Let $X$ be a smooth projective variety. Let $ L_1$ and $ L_2$ be two base-point free line bundles on $X$. 
    Assume the following: 
    \begin{enumerate}
    \renewcommand{\labelenumi}{\textup{(\roman{enumi})}}
        \item $M_{ L_2}\otimes  L_1$ is globally generated, and 
        \item $H^0( L_1)\otimes H^0( L_2)\to H^0( L_1+ L_2)$ is surjective.
    \end{enumerate}
    Then $M_{ L_1+ L_2}\otimes  L_1$ is globally generated.
\end{lemma}
\begin{proof}
    Note that the assumptions give rise to the following commutative diagram with exact rows and columns:
    \[
    \begin{tikzcd}
        & & 0\arrow[d] & 0\arrow[d] &\\
        & & H^0(M_{ L_2}\otimes  L_1)\otimes  L_1\arrow[r]\arrow[d] & M_{ L_2}\otimes (2 L_1)\arrow[d]\arrow[r] & 0\\
        0\arrow[r] & M_{ L_1}\otimes H^0( L_2)\otimes  L_1\arrow[r]\arrow[d] & H^0( L_1)\otimes H^0( L_2)\otimes  L_1\arrow[r]\arrow[d] & (2 L_1)\otimes H^0( L_2)\arrow[r]\arrow[d] & 0\\
        0\arrow[r] & M_{ L_1+ L_2}\otimes  L_1\arrow[r] & H^0( L_1+ L_2)\otimes  L_1\arrow[r]\arrow[d] & 2 L_1+ L_2\arrow[d]\arrow[r]& 0\\
        & & 0 & 0 &
    \end{tikzcd}
    \]
     Thus by snake lemma, the leftmost vertical map is also surjective. The conclusion follows since $$M_{ L_1}\otimes  L_1=\Lambda^{r-1}M_{ L_1}^*$$ is clearly globally generated where $r=\textrm{rank}(M_{ L_1})=h^0( L_1)-1$.
\end{proof}

\begin{corollary}\label{n+1}
    Let $X$ be a smooth projective variety of dimension $n$ and $L$ be an ample and base-point free line bundle on $X$. Then $M_{K_X+mL}\otimes {L}$ is globally generated for all $m\geq n+1$ if so is $M_{K_X+(n+1)L}\otimes L$.
\end{corollary}
\begin{proof}
    The assertion follows from Lemma \ref{l1} as $$H^0(K_X+mL)\otimes H^0(L)\to H^0(K_X+(m+1)L)$$ is surjective for any $m\geq n+1$ by Castelnuovo-Mumford regularity and Kodaira vanishing.
\end{proof}

In practice, to argue global generation of twisted syzygy bundles, we will argue through an inductive procedure by slicing the variety by hyperplane sections. While doing so, the following will be crucially used:

\begin{lemma}\label{resgg}
    Let $X$ be a smooth projective variety and $Z\subset X$ be a smooth subvariety. Let $ L_1$ and $ L_2$ be two globally generated line bundles on $X$. Assume:
    \begin{enumerate}
    \renewcommand{\labelenumi}{\textup{(\roman{enumi})}}
        \item $H^0( L_1)\to H^0( L_1|_Z)$ is surjective, 
        \item $M_{ L_1|_Z}\otimes  L_2|_Z$ is globally generated, 
        and
        \item $H^0( L_1|_Z)\otimes H^0( L_2|_Z)\to H^0(( L_1+ L_2)|_Z)$ is surjective.
    \end{enumerate}
    Then $M_{ L_1}\otimes L_2\otimes\mathcal{O}_Z$ is globally generated.
\end{lemma}
\begin{proof}
    Consider the commutative diagram with exact rows and columns where $\mathcal{I}_{Z/X}$ denotes the ideal sheaf of $Z\subset X$:
    \begin{equation}\label{eq1}
    \begin{tikzcd}
        & 0\arrow[d] & 0\arrow[d] & &\\
        & H^0( L_1\otimes\mathcal{I}_{Z/X})\otimes  L_2|_Z\arrow[r, equal]\arrow[d] & H^0( L_1\otimes\mathcal{I}_{Z/X})\otimes  L_2|_Z\arrow[d] & &\\
        0\arrow[r] & M_{ L_1}\otimes  L_2\otimes\mathcal{O}_Z\arrow[r]\arrow[d] & H^0( L_1)\otimes  L_2|_Z\arrow[r]\arrow[d] & ( L_1+ L_2)|_Z\arrow[r]\arrow[d, equal] & 0\\
        0\arrow[r] & M_{ L_1|_Z}\otimes  L_2|_Z\arrow[r]\arrow[d] & H^0( L_1|_Z)\otimes  L_2|_Z\arrow[r]\arrow[d] & ( L_1+ L_2)|_Z\arrow[r] & 0\\
        & 0 & 0 & &
    \end{tikzcd}
    \end{equation}
    In the above, the middle column is exact by our assumption (i). The middle two rows of the above induces the commutative diagram with exact rows
    \begin{equation*}
    \begin{tikzcd}
        0\arrow[r] & H^0(M_{ L_1}\otimes  L_2\otimes\mathcal{O}_Z)\arrow[r]\arrow[d] & H^0( L_1)\otimes H^0( L_2|_Z)\arrow[r, two heads]\arrow[d, two heads] & H^0(( L_1+ L_2)|_Z)\arrow[r]\arrow[d, equal] & 0\\
        0\arrow[r] & H^0(M_{ L_1|_Z}\otimes  L_2|_Z)\arrow[r] & H^0( L_1|_Z)\otimes H^0( L_2|_Z)\arrow[r, two heads] & H^0(( L_1+ L_2)|_Z)\arrow[r] & 0
    \end{tikzcd}
    \end{equation*}
    In the above, the indicated map in the bottom row is surjective by the assumption (iii), whence the indicated map in the top row is also surjective. Consequently, snake lemma shows that the left-most vertical arrow of the above diagram is also surjective. 
    Now the left-most vertical sequence of \eqref{eq1} via evaluation maps 
    gives rise to the following diagram with exact rows:
    \small
    \begin{equation*}
    \begin{tikzcd}
        0\arrow[r] & H^0( L_1\otimes\mathcal{I}_{Z/X})\otimes H^0( L_2|_Z)\otimes \mathcal{O}_Z\arrow[r]\arrow[d, two heads] & H^0(M_{ L_1}\otimes  L_2\otimes\mathcal{O}_Z)\otimes\mathcal{O}_Z\arrow[d]\arrow[r] & H^0(M_{ L_1|_Z}\otimes  L_2|_Z)\otimes\mathcal{O}_Z\arrow[d]\arrow[r] & 0\\
        0\arrow[r] & H^0( L_1\otimes\mathcal{I}_{Z/X})\otimes  L_2|_Z\arrow[r] & M_{ L_1}\otimes  L_2\otimes\mathcal{O}_Z\arrow[r] & M_{ L_1|_Z}\otimes  L_2|_Z\arrow[r] & 0
    \end{tikzcd}
    \end{equation*}
    \normalsize
    Note that the right-most vertical map is surjective by the assumption (ii), whence the assertion follows as the left-most vertical map is also surjective.
\end{proof}

We now proceed to discuss the positivity of twisted syzygy bundles on products. Let $X$ and $Y$ be two smooth projective varieties and let $p_1$ and $p_2$ be projections from $X\times Y$ to $X$ and $Y$ respectively. If $\mathcal{A}$ and $\mathcal{B}$ are coherent sheaves on $X$ and $Y$ respectively, we denote the sheaf $p_1^*\mathcal{A}\otimes p_2^*\mathcal{B}$ on $X\times Y$ by $\mathcal{A}\boxtimes\mathcal{B}$.

\smallskip

In this situation, we record the following result:

\begin{proposition}\label{prop2'}
Let $X$ and $Y$ be two smooth projective varieties. Let $\mathcal{A}$ (resp. $\mathcal{B}$) be a base-point free line bundle on $X$ (resp. $Y$). Then we have a short exact sequence
\begin{equation}\label{box1'}
            0\to M_{\mathcal{A}}\boxtimes (H^0(\mathcal{B})\otimes \mathcal{O}_Y)\to M_{\mathcal{A}\boxtimes\mathcal{B}}\to \mathcal{A}\boxtimes M_{\mathcal{B}}\to 0.
\end{equation}
In particular, if $M_\mathcal{A}\langle \delta L_1\rangle$ and $M_\mathcal{B}\langle \delta L_2\rangle$ are nef for some $\delta\in\mathbb{Q}_{>0}$ and ample line bundles $L_1$ and $L_2$ on $X$ and $Y$ respectively, then so is $M_{\mathcal{A}\boxtimes\mathcal{B}}\langle \delta (L_1\boxtimes L_2)\rangle$. 
\end{proposition}

\begin{proof}
        The exact sequence follows from the following diagram 
        \begin{equation*}
        \begin{tikzcd}
            0\arrow[r] & M_{\mathcal{A}}\boxtimes (H^0(\mathcal{B})\otimes \mathcal{O}_Y)\arrow[d]\arrow[r] & H^0(\mathcal{A})\otimes H^0(\mathcal{B})\otimes\mathcal{O}_{X\times Y}\arrow[d, equal]\arrow[r] & \mathcal{A}\boxtimes (H^0(\mathcal{B})\otimes \mathcal{O}_Y)\arrow[d]\arrow[r] & 0\\
            0\arrow[r] & M_{\mathcal{A}\boxtimes \mathcal{B}}\arrow[r] & H^0(\mathcal{A})\otimes H^0(\mathcal{B})\otimes\mathcal{O}_{X\times Y}\arrow[r] & \mathcal{A}\boxtimes \mathcal{B}\arrow[r] & 0
        \end{tikzcd}
        \end{equation*}
        using snake lemma.
The last assertion is a consequence of the formal properties of $\mathbb{Q}$-twisted sheaves.
\end{proof}

\subsection{Positivity of syzygy bundles for abelian, hyperelliptic and Kummer varieties}\label{subsec:nefness-abelian} Let us first recall some basic notions regarding positivity on abelian varieties.

Let $(X, L)$ be a polarized abelian variety of dimension $g$. Further, denote the dual abelian variety by $\hat{X}$ and the normalized Poincar\'e bundle on $X\times \hat{X}$ by $\mathcal{P}$. For $\alpha \in \hat{X}$, denote the corresponding line bundle on $X$ by $P_{\alpha}:=\mathcal{P}|_{X\times \left\{\alpha \right\}}$. For a coherent sheaf $\mathcal{F}$ on $X$, let $h^i_{\textrm{gen}}(\mathcal{F})$ be the dimension of the $H^i(\mathcal{F}\otimes P_{\alpha})$ for $\alpha\in\hat{X}$ general. Given any $b\in\mathbb{Z}^{>0}$, denote the multiplication by $b$ isogeny by $b_X:X\to X$.

In \cite{JP}, Jiang--Pareschi defined the {\it cohomological rank function} as follows $$h^i_{\mathcal{F}, L}:\mathbb{Q}\to\mathbb{Q}^{>0},\, x\mapsto h^i_{\mathcal{F}}(x L)\,\textrm{ where }\, h_{\mathcal{F}}^i(x L):=b^{-2g}h^i_{\textrm{gen}}(b_X^*\mathcal{F}\otimes (ab L))\,\textrm{ if }\, x=\frac{a}{b}, b>0.$$ We can think of $h^i_{\mathcal{F}}(x L)$ as the generic cohomological rank of the $\mathbb{Q}$-twisted sheaf $\mathcal{F}\langle x L\rangle$. Through these functions, {\it loc. cit.} defined the following  invariant:
$$\beta( L):=\inf\left\{x\in\mathbb{Q}\mid h^1_{\mathcal{I}_p}(x L)=0\right\}.$$ 
In the above, $p\in X$ is a point with ideal sheaf $\mathcal{I}_p$; the definition of the {\it base-point freeness threshold} $\beta( L)$ does not depend on the choice of $p\in X$. 
It turns out that $\beta( L)\in(0,1]$, and $\beta( L)<1$ if and only if $ L$ is base-point free (hence the name). Furthermore, 
\begin{equation}\label{rel}
    \beta(m L)=\frac{1}{m}\beta( L)\,\textrm{ for }m\geq 1.
\end{equation}

Let $X$ be an irregular variety with Albanese map ${\rm alb}_X:X\to A$. Given a coherent sheaf $\mathcal{F}$ on $X$, the  {\it cohomological support loci} $V^i(\mathcal{F})$ is defined as follows: $$V^i(\mathcal{F}):=\left\{\alpha\in \hat{A}\mid h^i(\mathcal{F}\otimes {\rm alb}_X^*(P_{\alpha}))\neq 0\right\}.$$
\begin{definition}
A coherent sheaf is called $\mathcal{F}$ is {\it GV} (resp. {\it M-regular}, {\it IT$_0$}) if $\codim_{\hat{A}} V^i(\mathcal{F})\geq i$ (resp. $\codim_{\hat{A}} V^i(\mathcal{F})> i$, $V^i(\mathcal{F})=\emptyset$) for all $i >0 $.
\end{definition}
In particular, we have the following chain of implications:
\begin{center}
    IT$_0$ $\implies$ M-regular $\implies$ GV.
\end{center}
We return to the abelian case. We again assume $(X, L)$ is a polarized abelian variety of dimension $g$ and $\mathcal{F}$ is a coherent sheaf on $X$. The definitions of GV, M-regularity and IT$_0$ sheaves extend to the $\mathbb{Q}$-twisted setting as follows: 
\begin{definition}
$\mathcal{F}\langle x L\rangle$ is {\it GV} (resp. {\it M-regular}, {\it IT$_0$}) if so is $b_X^*\mathcal{F}\otimes (ab L)$ where $x=\frac{a}{b}$, $b>0$.
\end{definition} 
It turns out that $\mathcal{F}\langle x L\rangle$ is GV if and only if $\mathcal{F}\langle (x+x') L\rangle$ is IT$_0$ for all rational $x'>0$ \cite[Theorem 5.2]{JP}. Moreover, we have
\begin{center}
    $\beta( L)<x\iff \mathcal{I}_o\langle x L\rangle$ is IT$_0$,
\end{center}
where $o \in X$ is the origin.

\begin{lemma}\label{gvn}
    Let $x\in\mathbb{Q}$, and assume $\mathcal{F}\langle x L\rangle$ is GV (resp. M-regular). Then $\mathcal{F}\langle x L\rangle$ is nef (resp. ample).
\end{lemma}
\begin{proof}
    Let $x=\frac{a}{b}$ with $b>0$. Consider the Cartesian square 
    \[
    \begin{tikzcd}
        \mathbb{P}(b_X^*\mathcal{F})\arrow[r, "\widetilde{b_X}"]\arrow[d, swap, "\widetilde{\pi}"] & \mathbb{P}(\mathcal{F})\arrow[d, "\pi"]\\
    X\arrow[r, "b_X"] & X
    \end{tikzcd}
    \]
    Note that we have $\widetilde{b_X}^*(\pi^*(x L))\equiv\widetilde{\pi}^*(b_X^*(x L))\equiv\widetilde{\pi}^*(ab L)$. Thus, we deduce $$\mathcal{O}_{\mathbb{P}(b_X^*\mathcal{F})}(1)\otimes(\widetilde{\pi}^*(ab L))\equiv\widetilde{b_X}^*(\mathcal{O}_{\mathbb{P}(\mathcal{F})}(1)\otimes(\pi^*(x L))).$$

    Now, when $\mathcal{F}\langle x L\rangle$ is GV (resp. M-regular), $\mathcal{O}_{\mathbb{P}(b_X^*\mathcal{F})}(1)\otimes(\widetilde{\pi}^*(ab L))$ is nef (resp. ample) by \cite[Theorem 4.1]{PaPo3} (resp. \cite[Corollary 3.2]{Deb}, keeping in mind that M-regular implies continuously globally generated, see \cite[Proposition 2.13]{PaPo1} for details).
    The assertions follow since $\widetilde{b_X}$ is finite.
\end{proof}

The following result is crucial for us:

\begin{proposition}[{\cite[Proposition 8.1]{JP}, \cite[Proposition 4.1]{MR4474904}}]\label{prop_JP_8.1}
    Assume $ L$ is base-point free.
    Then $\mathcal{I}_o \langle x L \rangle$ is GV (resp. M-regular, IT$_0$) if and only if so is $M_{ L}\langle \frac{x}{1-x}  L\rangle$ for a rational number $0 < x <1$.
\end{proposition}

\begin{corollary}\label{ab-g+1}
{ Assume $ L$ is base-point free and $m \geq g$. 
Then $ M_{m L}( L)$ is nef if $g =1$,  and ample if $g \geq 2$.}
\end{corollary}

\begin{proof}
{ Since 
$M_{m L}( L) = M_{m L}\langle \frac{x}{1-x}m L\rangle$
for $x =\frac{1}{m+1}$,
\begin{align}\label{eq_ab-g+1}
    \text{$M_{m L}( L) $ is GV (resp.\ M-regular)} \iff \text{so is $\mathcal{I}_o \langle \tfrac{m}{m+1}  L \rangle $}
\end{align}
by Proposition \ref{prop_JP_8.1}.

Assume $g=1$. By \eqref{eq_ab-g+1}, $\mathcal{I}_o \langle \frac{m}{m+1}  L \rangle $ is GV if and only if $ -1+ \frac{m}{m+1} \deg  L \geq 0$, that is,
$\deg  L \geq \frac{m+1}{m}$ by \cite[Example 2.1]{MR4474904}.
Since $\deg  L\geq 2$ by the base-point freeness of $  L$,
$\deg  L \geq \frac{m+1}{m}$ holds for any $m \geq g=1$.
Thus, for $m \geq 1$, $M_{m L}( L)$ is GV and hence nef by Lemma \ref{gvn}.
}

Assume $g \geq 2$ and let us consider the short exact sequence 
\[
0 \to M_{g L} \otimes  L \otimes P_{\alpha} \to H^0(g L)\otimes L \otimes P_{\alpha} \to (g+1) L \otimes P_{\alpha} \to 0.
\]
If $P_{\alpha}^g =P_{g \alpha} \not \simeq \mathcal{O}_X$,
then the map
$$H^0(g L)\otimes H^0( L \otimes P_{\alpha}) \to H^0((g+1) L\otimes P_{\alpha})$$ is surjective by Castelnuovo-Mumford regularity
and hence $$H^1(M_{g L} \otimes  L \otimes P_{\alpha}) =0. $$
Thus $M_{g L} \otimes  L$ is 
M-regular since $g \geq 2$.
Hence $\mathcal{I}_o \langle \frac{g}{g+1}  L \rangle $ is M-regular by \eqref{eq_ab-g+1}.
Then $\mathcal{I}_o \langle \frac{m}{m+1}  L \rangle $ is M-regular for $m \geq g$ by $\frac{m}{m+1} \geq \frac{g}{g+1}$.
Thus, for $m\geq g$, $M_{m L}( L)$ is M-regular by \eqref{eq_ab-g+1} and hence ample by Lemma \ref{gvn}.
\end{proof}

The second part of the remark below will be used in the proof of the first part of Theorem \ref{thm3'}.

\begin{remark}\label{ab-g+1'} \em{We also observe the following:
\begin{itemize}
    \item[$(1)$] By the above proof, we see that $\beta( L) \leq \frac{g}{g+1}$ if $L$ is base-point free.
    \item[$(2)$] Let $m\geq 2$ be an integer. If $ L$ is merely ample (and potentially has base-points), then $M_{m L}\langle a L\rangle$ is nef for any rational $a\geq \frac{m}{m-1}$. Indeed, since 
$M_{m L}\langle a L\rangle = M_{m L}\langle \frac{x}{1-x}m L\rangle$
for $x =\frac{a}{m+a}$, by Proposition \ref{prop_JP_8.1}
$M_{m L}(a L) $ is GV if and only if so is $\mathcal{I}_o \langle \tfrac{am}{m+a}  L \rangle $.
This holds if $ \tfrac{am}{m+a} \geq 1$, that is, $a \geq \frac{m}{m-1}$, whence the conclusion follows from Lemma \ref{gvn}.
\end{itemize}} 
\end{remark}

The notion of M-regularity is related to jet-ampleness for linear series on abelian varieties by \cite{PaPo3}.
We will use the following result:

\begin{theorem}[{\cite[Theorem 1.6]{MR4474904}}]\label{It}
    Let $L$ be an ample line bundle on an abelian variety $X$ and $k\geq 1$ an integer. If $\mathcal{I}_o\langle \frac{1}{k+1}L\rangle$ 
is M-regular, then L is $k$-jet ample.
\end{theorem}

Finally, we turn to hyperelliptic and Kummer varieties and recall the following recent results:

\begin{theorem}[{\cite[(5.3) on p. 16]{Cau25}}]\label{caucci}
    Let $G$ be a finite commutative group acting freely and not only by translations on an abelian variety $A$ of dimension $n$. Consider the quotient $\pi:A\to X:=A/G$. Let $L$ be an ample line bundle and  $N\equiv 0$ be a numerically trivial line bundle on $X$. Then 
    \begin{itemize}
        \item[$(1)$] The bundle $\pi^*(M_{N+mL}\otimes (2L))$ is GV for $m\geq 2n+1$. \footnote{This part follows easily by the arguments of \cite{Cau25}. This is pointed out to us by F. Caucci whom we thank.} 
        \item[$(2)$] If $q(X)>0$, then the same conclusion holds if $m\geq 2n$.
    \end{itemize}
\end{theorem}

\begin{lemma}[{\cite[Lemma 5.4]{MR4688553}}]\label{pos-kummer}
Let $K(A)$ be the Kummer variety associated to an abelian variety $A$.
Let $L$ be an ample line bundle on $K(A)$.
\footnote{$L$ is globally generated by \cite[Lemma 2.5]{MR4688553}.}
For $m \geq 2$,
$\pi^* M_{mL} \langle \frac{m}{2(m-1)} \pi^*L \rangle$ is IT$_0$.
\end{lemma}

\subsection{Positivity of syzygy bundles for regular varieties} 
In this subsection, we prove a proposition regarding the nefness of the syzygy bundle for regular varieties (i.e., varieties with trivial irregularity).

\begin{proposition}\label{prop1}
    Let $X$ be a smooth projective variety of dimension $n$ with $q(X)=0$. Let ${L}$ be a very ample line bundle on $X$.
    Let $\mathcal{C}_{L}$ be the class of smooth curves in $X$ for which there is a chain
    $$C=X_1\subset X_2\subset\cdots\subset X_n=X$$ such that each $X_i\in |L|_{X_{i+1}}|$ is smooth.
    For a curve $C$ in $\mathcal{C}_{L}$ we write $L_C:=L|_C$. 
    If $M_{K_C+2L_C}\otimes L_C$ is globally generated for every curve $C \in \mathcal{C}_{L}$, then $M_{K_X+mL}\otimes {L}$ is globally generated for all $m\geq n+1$.
\end{proposition}
\begin{proof}
    First, observe that it is enough to show that $M_{K_X+(n+1)L}\otimes L$ is globally generated by Corollary \ref{n+1}. To prove this, we use induction on $n$. Since $L$ is very ample, given any point $x\in X$, there is a smooth section $H\in|L|$ passing through $x$. Notice that $q(H)=0$ by restriction sequence and Kodaira vanishing when $n\geq 3$, whence to carry out the induction, it is enough to prove \begin{itemize}
        \item[(a)] the following map $$H^0(M_{K_X+(n+1)L}\otimes L)\to H^0(M_{K_X+(n+1)L}\otimes L\otimes\mathcal{O}_H)$$ arising from the restriction sequence is surjective, and 
        \item[(b)] If $M_{K_H+nL_H}\otimes L_H$ is globally generated, then so is $M_{K_X+(n+1)L}\otimes L\otimes\mathcal{O}_H$ where $L_H:=L|_H$.
    \end{itemize}
    
    We first prove (a). Using the defining exact sequence $$0\to M_{K_X+(n+1)L}\to H^0(K_X+(n+1)L)\otimes \mathcal{O}_X\to K_X+(n+1)L\to 0,$$ we see using $q(X)=0$ and the globally generation of $K_X +(n+1)L$ that $$H^1(M_{K_X+(n+1)L})=0,$$ whence the required surjection follows from the restriction sequence.
    
    We now prove (b) using Lemma \ref{resgg}. Recall that by adjunction, we have $(K_X+(n+1)L)\otimes\mathcal{O}_H=K_H+nL_H$. Notice that the restriction map $$H^0(K_X+(n+1)L)\to H^0(K_H+nL_H)$$is surjective by Kodaira vanishing. Lastly, the multiplication map $$H^0(K_H+nL_H)\otimes H^0(L_H)\to H^0(K_H+(n+1)L_H)$$ is surjective by Castelnuovo-Mumford regularity and Kodaira vanishing. The proof is now complete.
\end{proof}

The following is the main result of this subsection:

\begin{theorem}\label{thm1}
    Let $(X,L)$ be a polarized smooth regular projective variety of dimension $n$ with $L$ very ample.  Then $M_{K_X+mL}\otimes L$ is globally generated for all $m\geq n+1$.
\end{theorem}

\begin{proof} By Corollary \ref{n+1}, we may assume $m=n+1$.
Furthermore, we may assume $(X,L)\ncong(\mathbb{P}^n,\mathcal{O}_{\mathbb{P}^n}(1))$ and invoke Proposition \ref{prop1}. 
Let $$C=X_1\subset X_1\subset \cdots\subset X_n=X$$ be a chain as in the statement of Proposition \ref{prop1}, and we aim to show that $M_{K_C+2L_C}\otimes L_C$ is globally generated. Let $p\in C$ be a point. By the restriction sequence, it is enough to show that 
    \begin{equation}\label{tp1}
        \textrm{Ker}\left(H^1(M_{K_C+2L_C}\otimes (L_C(-p)))\to H^1(M_{K_C+2L_C}\otimes L_C)\right)=0.
    \end{equation}
    The map above fits into the following commutative diagram 
    \begin{equation}\label{diag1}
    \begin{tikzcd}
        H^1(M_{K_C+2L_C}\otimes (L_C(-p)))\arrow[r]\arrow[d] & H^0(K_C+2L_C)\otimes H^1(L_C(-p))\arrow[d]\\
        H^1(M_{K_C+2L_C}\otimes L_C)\arrow[r] & H^0(K_C+2L_C)\otimes H^1(L_C)
    \end{tikzcd}
    \end{equation}
    where the right vertical map is also induced by the restriction sequence.
    
    We use the exact sequence 
    \begin{equation}\label{mex1}
        0\to M_{K_C+2L_C}\otimes (L_C(-p))\to H^0(K_C+2L_C)\otimes (L_C(-p))\to (K_C+3L_C(-p))\to 0.
    \end{equation}
    Note that the evaluation map $$H^0(K_C+2L_C)\otimes H^0(L_C(-p))\to H^0(K_C+3L_C(-p))$$ is surjective. Indeed, this follows from Castelnuovo-Mumford regularity and Kodaira vanishing as $L_C(-p)$ is ample and globally generated since we assumed $(X,L)\ncong(\mathbb{P}^n,\mathcal{O}_{\mathbb{P}^n}(1))$. Consequently, the long exact sequence of \eqref{mex1} shows that the top horizontal map of \eqref{diag1} is injective. On the other hand, the restriction sequence shows that the right vertical map of \eqref{diag1} is also injective as $L_C$ is globally generated. Consequently the left vertical map of \eqref{diag1} is injective, which proves \eqref{tp1}.
\end{proof}

\subsection{Positivity of syzygy bundles for smooth fibrations over abelian varieties} We now study the positivity of the twisted syzygy bundles for smooth fibrations over abelian varieties. Although at first glance it might look too technical, however, it simplifies tremendously when the tangent bundle $\mathcal{T}_X$ of $X$ is nef, as we will see in the next subsection.

\begin{theorem}\label{theorem_nefness}
Let $\pi: X\to A$ be a smooth fibration over an abelian variety $A$. 
Let  $F_p$ be the fiber of $\pi$ over $p\in A$. 
Let $L$ and $H$ be ample line bundles on $X$ with the following properties:
    \begin{enumerate}
    \renewcommand{\labelenumi}{\textup{(\roman{enumi})}}
    \item $L$ is globally generated, GV, 
    and $\codim_{\hat{A}} \{\alpha \in \hat{A}\mid L+\pi^*P_{\alpha}\textrm{ is not globally generated } \} \geq 1$. 
    \item $H-nL$ is nef and big for $n:=\dim X$,
        \item 
        $H^i(L|_{F_p})=0$ for all $p\in A$ and $i>0$,
        \item there exists an ample line bundle $L_A$ on $A$ such that $H -\beta(L_A)\pi^* L_A$ is ample, and
        \item $M_{K_{F_p} +H|_{F_p}} \otimes L|_{F_p} $ is globally generated for any $p \in A$.
    \end{enumerate}
    Then $K_X+H$ is globally generated and  $M_{K_X+H}\otimes L$ is nef. 
\end{theorem}

\begin{proof} Note that $K_X+H$ is globally generated by Castelnuovo-Mumford regularity, Kawamata-Viehwag vanishing and the assumption (ii). 

Let us first prove the following two claims.

\begin{claim}\label{claim1.1}
The restriction map $H^0(K_X+ H) \to H^0((K_X+ H)|_{F_p})$ is surjective for all $p\in A$. 
\end{claim}

\begin{proof}
By $K_X|_{F_p} \cong K_{F_p}$ and Kodaira vanishing,  we have
\begin{itemize}
\item $ \pi_*(K_X+H)$ is locally free,
\item $R^i\pi_*(K_X+H)=0$ for all $ i>0$,
\item $ \pi_*(K_X+H) \otimes k(p) \cong H^0((K_X + H) |_{F_p} ) =H^0(K_{F_p} + H |_{F_p} ) $,
\item  $ H^0(\pi_*(K_X+H))= H^0(K_X+H)$.
\end{itemize}
Hence the restriction map $H^0(K_X+ H) \to H^0((K_X+ H)|_{F_p})$ coincides with the natural map
\[
H^0( \pi_*(K_X+H) ) \to  \pi_*(K_X+H) \otimes k(p).
\]
Thus we need to show that $ \pi_*(K_X+H)$ is globally generated.

By the condition (iv), we can take a rational number $t=\frac{a}{b} \geq  \beta(L_A)$  so that $ H -t \pi^* L_A$ is ample.
By \cite[Theorem 1.2 (1)]{MR4474904},
it suffices to show that $ \pi_*(K_X+H) \langle -t L_A \rangle $ is IT$_0$.
Consider the Cartesian square
    \[
    \begin{tikzcd}
        \widetilde{X}\arrow[r, "\widetilde{b_A}"]\arrow[d, swap, "\widetilde{\pi}"] & X\arrow[d, "\pi"]\\
    A\arrow[r, "b_A"] & A
    \end{tikzcd}
    \]
Since $b_A$ and $\widetilde{b_A}$ are \'etale, it holds that
\begin{align*}
b_A^*  \pi_*(K_X+H) \otimes (-ab L_A) \otimes P_\alpha &= \widetilde{\pi}_* \widetilde{b_A}^*(K_X+H)  \otimes (-ab L_A) \otimes P_\alpha  \\
&=  \widetilde{\pi}_* (K_{\widetilde{X}} + \widetilde{b_A}^*H)  \otimes (-ab L_A)\otimes P_\alpha  \\
&= \widetilde{\pi}_* (K_{\widetilde{X}} + \widetilde{b_A}^*H -ab\widetilde{\pi}^* L_A\otimes \widetilde{\pi}^*P_\alpha )
\end{align*}
for $\alpha \in \hat{A}$.
Since 
\[
\widetilde{b_A}^*H -ab\widetilde{\pi}^* L_A\otimes \widetilde{\pi}^*P_\alpha \equiv \widetilde{b_A}^*H -ab\widetilde{\pi}^* L_A = \widetilde{b_A}^*H - t \widetilde{\pi}^* b^2 L_A  \equiv \widetilde{b_A}^*H - t \widetilde{\pi}^*b_A^*L_A =\widetilde{b_A}^*(H -t \pi^* L_A) 
\]
is ample,
we have
\[
R^i \widetilde{\pi}_* (K_{\widetilde{X}} + \widetilde{b_A}^*H -ab\widetilde{\pi}^* L_A\otimes \widetilde{\pi}^*P_\alpha )=0\,\textrm{ for all }i>0.
\]
Thus, the Leray spectral sequence and Kodaira vanishing give
\begin{align*}
H^i(b_A^*  \pi_*(K_X+H) \otimes (-ab L_A) \otimes P_\alpha)&=H^i(\widetilde{\pi}_* (K_{\widetilde{X}} + \widetilde{b_A}^*H -ab\widetilde{\pi}^* L_A\otimes \widetilde{\pi}^*P_\alpha) ) \\
&\cong H^i(K_{\widetilde{X}} + \widetilde{b_A}^*H -ab\widetilde{\pi}^* L_A\otimes \widetilde{\pi}^*P_\alpha )=0
\end{align*}
for all $i >0$,
which implies that $b_A^*  \pi_*(K_X+H) \otimes (-ab L_A) $ is IT$_0$.
Hence $ \pi_*(K_X+H) \langle -t L_A \rangle $ is IT$_0$.
\end{proof}

\begin{claim}\label{claim1.2}
The sheaf $\pi_*(M_{K_X +H} \otimes L)$ is locally free and  nef.
\end{claim}

\begin{proof}
For a line bundle $L' =L + \pi^* P_{\alpha}$,
consider the exact sequence 
\[
0\to (M_{K_X+H}\otimes L')|_{F_p}\to H^0(K_X+H)\otimes L'|_{F_p}\to K_{F_p}+H|_{F_p}+L'|_{F_p}\to 0
\]
on $F_p$.
Now, the map $H^0(K_X+H)\otimes H^0(L'|_{F_p})\to H^0(K_{F_p}+H|_{F_p}+L'|_{F_p})$ is surjective since it factors as 
\begin{align}\label{eq_claim1.2}
    H^0(K_X+H)\otimes H^0(L'|_{F_p})\to H^0(K_{F_p}+H|_{F_p})\otimes H^0(L'|_{F_p})\to  H^0(K_{F_p}+H|_{F_p} + L'|_{F_p})
\end{align}
 and the first (resp.\ second) map is surjective by Claim \ref{claim1.1} (resp.\ Castelnuovo-Mumford regularity and Kodaira vanishing by the condition (ii) since $L'|_{F_p} =L|_{F_p}$ is globally generated. We note that $H|_{F_p} -(\dim F_p) L'|_{F_p} = (H-nL')|_{F_p} +(n -\dim F_p) L'|_{F_p}  $ is nef and big.). 
 Since $$H^i(L'|_{F_p})=H^i(L|_{F_p})=0\textrm{ for $i>0$}$$  by condition (iii), 
we conclude that 
\begin{align}\label{eq1.2}
H^i((M_{K_X+H}\otimes L')|_{F_p})=0\,\textrm{ for all }i>0.
\end{align} 
 Hence $$\pi_*(M_{K_X+H}\otimes L')=\pi_*(M_{K_X+H}\otimes L) \otimes \pi^* P_\alpha  \textrm{ is locally free, and } R^i\pi_*(M_{K_X+H}\otimes L')=0\textrm{ for }i>0.$$
Then the Leray spectral sequence shows that 
\[
H^i(\pi_*(M_{K_X+H}\otimes L)\otimes P_{\alpha}) =H^i( \pi_*(M_{K_X+H}\otimes L')) \cong H^i(M_{K_X+H}\otimes L') =H^i(M_{K_X+H}\otimes L\otimes\pi^*P_{\alpha}).
\]
 Now the exact sequence 
 \[
 0\to M_{K_X+H}\otimes L\otimes\pi^* P_{\alpha}\to H^0(K_X+H)\otimes L\otimes\pi^*P_{\alpha}\to K_X+H+L+\pi^*P_{\alpha}\to 0
\] 
shows that
\begin{itemize}
\setlength{\itemsep}{0mm}
\item $\left\{\alpha\mid H^1(M_{K_X+H}\otimes L\otimes \pi^*P_{\alpha})\neq 0\right\}$ is contained in 
\[
\left\{\alpha\mid L+\pi^*P_{\alpha}\textrm{ is not globally generated }\right\}\cup \left\{\alpha\mid H^1(L+\pi^*P_{\alpha})\neq0\right\}
\]
 by the condition (ii) and  Castelnuovo-Mumford regularity, and
\item $\left\{\alpha\mid H^i(M_{K_X+H}\otimes L\otimes \pi^*P_{\alpha})\neq 0\right\}= \left\{\alpha\mid H^i(L+\pi^*P_{\alpha})\neq0\right\}$ for $i \geq 2$.
\end{itemize}
        Consequently, by (i), the vector bundle $\pi_*(M_{K_X+H}\otimes L) $ is GV and hence nef by \cite[Theorem 4.1]{PaPo3}.
\end{proof}

\begin{claim}\label{claim1.3}
The natural map 
$\pi^* \pi_*(M_{K_X+H} \otimes L) \to M_{K_X+H} \otimes L $ is surjective.
\end{claim}

\begin{proof}
To show that  $\pi^* \pi_*(M_{K_X+H} \otimes L) \to M_{K_X+H} \otimes L $ is surjective,
it suffices to show that  $$\pi^* \pi_*(M_{K_X+H} \otimes L)|_{F_p} \to (M_{K_X+H} \otimes L)|_{F_p} \textrm{  is surjective for any $p \in A$. }$$
Since 
\[
\pi^* \pi_*(M_{K_X+H} \otimes L)|_{F_p} =   \pi_*(M_{K_X+H} \otimes L) \otimes k(p) \otimes \mathcal{O}_{F_p} =H^0((M_{K_X+H} \otimes L)|_{F_p} )\otimes \mathcal{O}_{F_p} 
\]
by \eqref{eq1.2},
it suffices to show that $(M_{K_X+H} \otimes L)|_{F_p}$ is globally generated.
By Lemma \ref{resgg},
it suffices to check
\begin{itemize}
\item $H^0(K_X +H) \to H^0((K_X +H)|_{F_p})  $ is surjective,
\item $M_{(K_X +H)|_{F_p}} \otimes L|_{F_p} =M_{K_{F_p} +H|_{F_p}} \otimes L|_{F_p} $ is globally generated,
\item $H^0( K_{F_p} +H|_{F_p} ) \otimes H^0(L|_{F_p} ) \to H^0( K_{F_p} +H|_{F_p} +L|_{F_p} ) $ is surjective.
\end{itemize}
The first one follows from Claim \ref{claim1.1} and
the second one is nothing but the condition (v).
The last one holds since this is the surjectivity of the second map in \eqref{eq_claim1.2}
for $P_\alpha=\mathcal{O}_A$, which we have already checked.
\end{proof}

We continue with the proof of Theorem \ref{theorem_nefness}. Since a pullback and a quotient of a nef vector bundle are nef,
it suffices to show  
\begin{itemize}
\item $\pi_*(M_{K_X+H}\otimes L)$ is a nef vector bundle, and
\item $\pi^* \pi_*(M_{K_X+H}\otimes L) \to M_{K_X+H}\otimes L $ is surjective.
\end{itemize}
Thus, the conclusion follows from Claims \ref{claim1.2} and \ref{claim1.3}.
\end{proof}

\subsection{Positivity of syzygy bundles for varieties with nef tangent bundles} We record an useful consequence of the results we have proven thus far. The main result of this subsection is the following:
\begin{corollary}\label{betacor}
    Let $X$ be a smooth irregular projective variety of dimension $n$ with nef tangent bundle $\mathcal{T}_X$. Assume that the fibers of the Albanese morphism of $X$ denoted by ${\rm alb_X}:X\to A$ are regular. Let $L$ be a very ample line bundle on $X$ and $m\geq n+1$ be an integer. If there exists an ample line bundle $L_A$ on $A$ such that $mL-\beta(L_A){\rm alb}_X^*(L_A)$ is ample, then $M_{K_X+mL}\otimes L$ is nef. 
\end{corollary}
\begin{proof}
    For $p\in A$, note that the fiber $F_p$ of the Albanese map is smooth and irreducible with nef tangent bundle $\mathcal{T}_{F_p}$ (whence $-K_{F_p}$ is nef) by Theorem \ref{structure}. 
    We check the assumptions (i)-(v) in Theorem \ref{theorem_nefness} for $H=mL$.
    By assumption, (ii), (iv) are clear and (iii) holds by Kodaira vanishing.
    (v) follows from Theorem \ref{thm1} as $m\geq \dim F_p+2$.
    For (i), $L$ is IT$_0$ and hence GV by Kodaira vanishing.
    Furthermore, $\codim_{\hat{A}} \{\alpha \in \hat{A}\mid L+\pi^*P_{\alpha}\textrm{ is not globally generated } \} \geq 1$ holds.
    To see this, consider the line bundle
    \[
    \mathscr{L} = p_1^* L \otimes (\pi\times \id)^* \mathcal{P},
    \]
    on $X \times \hat{A}$, where $p_1 : X \times \hat{A} \to X $ is the natural projection
    and $\mathcal{P} $ is the normalized Poincar\'e bundle on $A \times \hat{A}$.
    For the natural projection $p_2 :X \times \hat{A} \to \hat{A} $, ${p_2}_* \mathscr{L}$ is locally free and $ {p_2}_* \mathscr{L} \otimes k(\alpha) =H^0(L \otimes \pi^* P_\alpha)$ for any $\alpha$
    since $L$ is IT$_0$. 
    Hence 
    \begin{align*}
     Z &\coloneqq \{ (x,\alpha) \in X \times \hat{A} \mid \text{$H^0(L \otimes P_\alpha) \to L \otimes P_\alpha \otimes k(x)$ is not surjective}\}\\
     &= \{ (x,\alpha) \in X \times \hat{A} \mid \text{$p_2^* {p_2}_* \mathscr{L} \to \mathscr{L} \otimes k(x,\alpha) $ is not surjective}\}
    \end{align*}
    is a closed subset of $X \times \hat{A}$.
    Thus $p_2(Z) =\{\alpha \mid \text{$ L \otimes P_\alpha$ is not globally generated }\} \subset \hat{A}$ is a closed subset. Since $L$ is globally generated, 
    the origin $\hat{o} \in \hat{A}$ is not contained in $p_2(Z)$ and hence $ \codim_{\hat{A}} p_2(Z) \geq 1 $.
\end{proof}

\begin{remark}\label{rmk-va}
{\em 
In Corollary \ref{betacor},
it suffices to assume that $L|_{F_p}$ is very ample for all $p\in A$, rather than $L$ itself being very ample.}
\end{remark}

\section{From positivity of syzygy bundles to hyperbolicity}\label{transition}

Let $X$ be a smooth projective variety of dimension $n$.
Let $\mathcal{E}$ be a globally generated line bundle on $X$, and let $ L$ be an ample line bundle on $X$.
We are interested in studying the curves contained in a very general section of $\mathcal{E}$.
Setting $B_1 = H^0(X,\mathcal{E})$ and letting $\mathcal{Y}_1 \to B_1$ denote the universal hypersurface, we will study \textit{versal} families of curves in $\mathcal{Y}_1 \to B_1$. 
The terminology, used in \cite{ClRa04}, suggests that a curve in such a family deforms with every local deformation of the hypersurface $Y$ in $B_1$.
The following setup was already considered in some of the papers \cite{Ein88, Voi96, Voi98, Pac04, CoRi1, Yeo24} cited in the introduction; here, we will roughly follow the presentation in \cite{CoRi2, MY24}.

\begin{construction}\label{construction}
{\em
Let $B_1 = H^0(X,\mathcal{E})$, and let $\mathcal{Y}_1\to B_1$ be the universal hypersurface. 
Suppose a general section $Y \subset X$ of $\mathcal{E}$ contains a curve of geometric genus $g$ and degree $e$ with respect to $ L$.
Let $\psi:\mathcal{H}\rightarrow B_1$ be the relative Hilbert scheme parametrizing curves of geometric genus $g$ and degree $e$ with respect to $ L$, which is a dominant map.
Let $\mathcal{Y}_2 := \mathcal{Y}_1\times_{B_1} \mathcal{H}$, and let $\mathcal{C}_1\rightarrow \mathcal{H}$ denote the universal curve, so that we have 
\[
\xymatrix{
\mathcal{C}_1\, \ar@{^{(}->}[rr] \ar[dr] && \mathcal{Y}_2 \ar[dl]\\
 &\mathcal{H}&
}.
\]
We replace $\mathcal{H}$ with a subvariety $\mathcal{H}_1 \subset \mathcal{H}$ such that $\psi\vert_{\mathcal{H}_1}:\mathcal{H}_1\rightarrow B_1$ is \'etale, and denote the pullback of the above diagram over $\mathcal{H}_1$ by
\[
\xymatrix{
\mathcal{C}_2\, \ar@{^{(}->}[rr] \ar[dr] && \mathcal{Y}_3. \ar[dl]\\
 &\mathcal{H}_1&
}
\]
We take a resolution of $\mathcal{C}_2 \to \mathcal{H}_1$ and, after possibly restricting $\mathcal{H}_1$ to an open subset $B \subset \mathcal{H}_1$, obtain a smooth family $\mathcal{C} \to B$. 
This restriction is also to ensure that the pullback $\mathcal{Y}_3 \to \mathcal{H}_1$ to $\mathcal{Y} \to B$ is a smooth family.

To summarize, we obtain the following diagram
\[
\xymatrix{
\mathcal{C}\, \ar[rr]^-{g} \ar[dr] && \mathcal{Y} \ar[dl]\\
 &B &
}
\]
where $g\colon \mathcal{C} \to \mathcal{Y}$ is a generically injective map.
We denote by $\pi_1:\mathcal{Y} \rightarrow B$ and $\pi_2: \mathcal{Y} \rightarrow X$ the natural projection maps.
For each $b\in B$, we denote the restriction of $g$ over $b$ by $g_b:C_b\to Y_b$, and the inclusion of the fiber of $\mathcal{C}\to B$ over $b\in B$ by $\iota_b:C_b\to \mathcal{C}.$
}

\end{construction}

\begin{notation}\label{notation}
{\em
We introduce some notation regarding relative tangent and normal sheaves coming from Construction~\ref{construction}:
\begin{enumerate}
\item $N_{g/\mathcal{Y}}$ is the normal sheaf for $g:\mathcal{C}\to \mathcal{Y},$ defined as the quotient in the short exact sequence:
\begin{equation*}
    0 \to \mathcal{T}_\mathcal{C} \to g^*\mathcal{T}_\mathcal{Y} \to N_{g/\mathcal{Y}}\to 0.
\end{equation*}
\item $\mathcal{T}_{\mathcal{Y}/X}$ denotes the kernel of $\mathcal{T}_\mathcal{Y} \to \pi_2^*\mathcal{T}_X$, which is surjective since $\mathcal{E}$ is globally generated.
\item We denote the kernel and image of $\mathcal{T}_\mathcal{C} \to g^*\pi_2^*\mathcal{T}_X$ by $\mathcal{T}_{\mathcal{C}/X}$ and $\mathcal{Q}_1$, respectively.
\item $\mathcal{Q}_2$ denotes the cokernel of the natural inclusion $\mathcal{T}_{\mathcal{C}/X}\to g^*\mathcal{T}_{\mathcal{Y}/X}$.
\item $\mathcal{Q}_3$ denotes the cokernel of the natural inclusion $\mathcal{Q}_2\to N_{g/\mathcal{Y}}.$
\end{enumerate}
}
\end{notation}

To summarize, the above maps and sheaves fit into the following exact commutative diagram:

\begin{equation}\label{equation:setup}
\begin{tikzcd}
& &0 \arrow[d] &0 \arrow[d] &0 \arrow[d] &\\
&0 \arrow[r] &\mathcal{T}_{\mathcal{C}/X} \arrow[r] \arrow[d] &g^*\mathcal{T}_{\mathcal{Y}/X} \arrow[r] \arrow[d] &\mathcal{Q}_2 \arrow[r] \arrow[d] &0 \\
&0 \arrow[r] &\mathcal{T}_{\mathcal{C}} \arrow[r] \arrow[d] &g^*\mathcal{T}_{\mathcal{Y}} \arrow[r] \arrow[d] &N_{g/\mathcal{Y}} \arrow[r] \arrow[d] &0 \\
&0 \arrow[r] &\mathcal{Q}_1 \arrow[r] \arrow[d] &g^*\pi_2^*\mathcal{T}_{X} \arrow[r] \arrow[d] &\mathcal{Q}_3 \arrow[r] \arrow[d] &0 \\
& &0 &0  &0 &
\end{tikzcd}    
\end{equation}

\begin{remark}\label{remark:almost-homogeneous}
{\em 
When $X$ contains a Zariski-open set $X_0$ admitting a transitive group action by an algebraic group and for a curve not contained in $X\setminus X_0$, the above construction can be made so that 
$\rank\mathcal{Q}_1=n$, $\rank\mathcal{Q}_2=n-2$, and $\mathcal{Q}_3$ is torsion with $\iota_b^*\mathcal{Q}_3$ supported on $C_b\setminus X_0$ for each $b\in B$. 
When $X$ is homogeneous, we may assume that $\mathcal{Q}_1\simeq g^*\pi_2^*\mathcal{T}_X$, $\mathcal{Q}_2\simeq N_{g/\mathcal{Y}}$ and $\mathcal{Q}_3$ is zero.
(see, e.g., \cite[\S2.1]{CoRi2} and \cite[Construction~2.3]{MY24}).
}
\end{remark}

We recall the following

\begin{proposition}[{c.f.~\cite[Proposition 2.1]{CoRi2}}]\label{proposition:cori}
In the setting of Construction~\ref{construction} and Notation~\ref{notation},
the following statements hold:
\begin{enumerate}
    \item $\pi_2^*M_\mathcal{E}\simeq \mathcal{T}_{\mathcal{Y}/X}.$
    \item For every $b\in B$, $N_{g/\mathcal{Y}}\vert_{C_b} \simeq N_{g_b/Y_b}$.
\end{enumerate}
\end{proposition}

\subsection{Almost homogeneous varieties.}

In this subsection $X$ denotes a Gorenstein canonical projective variety containing a Zariski-open set admitting a transitive group action by an algebraic group.
The goal of this subsection is to prove:

\begin{proposition}\label{proposition:almost-homogeneous}
    Let $X$ be a Gorenstein canonical projective variety of dimension $n$ containing a Zariski-open set $X_0$ admitting a transitive group action by an algebraic group $G$, with a resolution $\rho:X'\to X$ that is an isomorphism over $X_0$.
    Let $Y\subset X$ be a hypersurface that is general in the linear system of a globally generated line bundle $\mathcal{E}$ and $f: C\to Y$ be a generically injective map from a smooth projective curve $C$ such that $f(C)\not\subset X\setminus X_0$.
    Denote by $\mathcal{E}':=\rho^*\mathcal{E}$, $Y'\subset X'$ the closure of $Y\cap X_0$ in $X'$, and $f':C\to Y'$ the lift of $f$ to its closure in $X'.$
    Apply Construction~\ref{construction} and Remark~\ref{remark:almost-homogeneous} with respect to $X'$, $\mathcal{E}'$ so that $\mathcal{H}$ contains a point $[f']$ corresponding to $f' : C \to Y'$, and $B$ and $[f']$ are contained in a common irreducible component of $\mathcal{H}$.
    Then for some $b\in B$, there is a surjection 
    \[ M_{\mathcal{E}'}\vert_{C_b} \to Q, \]
    where $Q$ is a sheaf on $C_b$ with $\deg N_{f/Y}\geq \deg Q.$
\end{proposition}

\begin{proof}
By Proposition~\ref{proposition:cori}, \eqref{equation:setup} restricts to the following exact diagram of sheaves on $C_b$ for any $b\in B$:

\begin{equation}\label{equation:setup-restricted}
\begin{tikzcd}
&0 \arrow[r] &g_b^*\pi_2^*M_{\mathcal{E}'} \arrow[r] \arrow[d] &g_b^*\mathcal{T}_{\mathcal{Y}} \arrow[r] \arrow[d] &g_b^*\pi_2^*\mathcal{T}_{X'} \arrow[r] \arrow[d] &0 \\
&&\iota_b^*\mathcal{Q}_2 \arrow[r, "\alpha_b"] \arrow[d] &N_{g_b/Y_b'} \arrow[r] \arrow[d] &\iota_b^*\mathcal{Q}_3 \arrow[r] \arrow[d] &0 \\
& &0 &0  &0 &
\end{tikzcd}    
\end{equation}

Since $\iota_b^*\mathcal{Q}_3 $ is torsion, we have $\deg \iota_b^*\mathcal{Q}_3 \geq0.$
Therefore,
$\deg N_{g_b/Y'_b}\geq \deg Q$, where $Q$ denotes the image sheaf of $\alpha_b.$
By \eqref{equation:setup-restricted}, the composite map $g_b^*\pi_2^*M_{\mathcal{E}'} \to \iota_b^*\mathcal{Q}_2  \to Q $ is surjective.
Since $X$ is Gorenstein canonical,
\[\deg N_{g_b/Y'_b} {= 2g(C_b)-2-(K_{X'}+\mathcal{E}')\cdot C_b}= 2g(C)-2-(K_{X'}+\mathcal{E}')\cdot C \leq 2g(C)-2-(K_X+\mathcal{E})\cdot C = \deg N_{f/Y}.\]
The proof is now complete.
\end{proof}

\subsection{Varieties with (pseudo or almost) nef tangent bundle.}
In this subsection, we prove a corresponding result for varieties whose tangent bundle is (pseudo or almost) nef, in the sense of Definitions \ref{def:gen-nef}, \ref{pnef}.

\begin{proposition}\label{proposition:nef-tangent}
    Let $X$ be a Gorenstein canonical projective variety of dimension $n$ whose tangent sheaf $\mathcal{T}_X$ is nef outside of $\cup_{i\in\mathbb{N}}Z_i$ where $Z_i\subsetneq X$ are proper closed subvarieties, with a resolution $\rho:X'\to X$ that is an isomorphism over $X\setminus \rm Sing(X)$.
    Let $Y\subset X$ be a hypersurface that is general in the linear system of a globally generated line bundle $\mathcal{E}$ and $f: C\to Y$ a generically injective map from a smooth projective curve $C$ such that $f(C)\not\subset\bigcup_{i\in\mathbb{N}}Z_i\cup{\rm Sing}(X)$.
Denote by $\mathcal{E}':=\rho^*\mathcal{E}$, $Y'\subset X'$ the closure of $Y\cap X_0$ in $X'$, and $f':C\to Y'$ the lift of $f$ to its closure in $X'.$
Apply Construction~\ref{construction} with respect to $X'$, $\mathcal{E}'$ so that $\mathcal{H}$ contains a point $[f']$ corresponding to $f' : C \to Y'$, and $B$ and $[f']$ are contained in a common irreducible component of $\mathcal{H}$.
    Then for some $b\in B$, there is a surjection 
    \[ M_{\mathcal{E}'}\vert_{C_b} \to Q, \]
    where $Q$ is a sheaf with $\deg N_{f/Y}\geq \deg Q.$
\end{proposition}

\begin{proof}
As in the proof of Proposition~\ref{proposition:almost-homogeneous}, we have the restricted exact diagram \eqref{equation:setup-restricted} of sheaves on $C_b.$
Since $g_b^*\pi_2^*\mathcal{T}_{X'}$ is nef on $C_b$, $\deg \iota_b^*\mathcal{Q}_3  \geq 0$, and so $\deg N_{g_b/Y'_b}\geq \deg Q$, where $Q$ denotes the image sheaf of $\alpha_b.$
As in Proposition~\ref{proposition:almost-homogeneous}, $X$ being Gorenstein canonical implies that $\deg N_{f/Y}\ge\deg N_{g_b/Y'_b}.$
\end{proof}

\subsection{The fundamental theorems.}
In this subsection, we prove the basic theorems that we use to prove our main results. The first one is for varieties with (pseudo, almost) nef tangent bundle.

\begin{theorem}\label{thm:hyperbolicity-of-N+L}
    Let $X$ be a Gorenstein canonical projective variety of dimension $n\geq 3$ with $\mathcal{T}_X$ nef outside of $\bigcup_{i\in\mathbb{N}}Z_i$ where $Z_i\subsetneq X$ are proper closed subvarieties. 
    Let $\mathcal{E}$ be a globally generated line bundle and $L$ be an ample line bundle on $X$.
    Assume that there exists a rational number $\delta>0$ such that the following conditions are satisfied:
    \begin{enumerate}
    \renewcommand{\labelenumi}{\textup{(\roman{enumi})}}
        \item[(i)] $M_{\mathcal{E}}\langle \delta L\rangle$ is nef, and
        \item[(ii)] $K_X+\mathcal{E}-\delta(n-2)L$ is ample.
    \end{enumerate}
    Then $|\mathcal{E}|$ is a hyperbolic linear system outside of $\bigcup_{i\in\mathbb{N}}Z_i\cup \operatorname{Sing}(X)$.
\end{theorem}

\begin{proof}
Let $Y\subset X$ be a very general hypersurface in $\left|\mathcal{E}\right|$, and $f:C\rightarrow Y$ a generically injective map from a smooth projective curve $C$ such that $f(C)\not\subset \bigcup_{i\in\mathbb{N}}Z_i\cup \operatorname{Sing}(X)$.
Denote by $\rho: X' \to X$ a resolution of $X$ which is an isomorphism over $X\setminus \operatorname{Sing}(X)$, $\mathcal{E}':=\rho^*\mathcal{E}$, $L':=\rho^*L$, $Y'\subset X'$ the closure of $Y\cap (X\setminus \operatorname{Sing}(X))$ in $X'$, and $f':C \to X'$ the lift of $f$ to its closure in $X'$.
Since we assume that $M_{\mathcal{E}}\langle \delta L\rangle$ is nef, 
$M_{\mathcal{E}'}\langle\delta L'\rangle$ is pseudo-nef outside of $\rho^{-1}({\rm Sing}(X))$.
By Proposition~\ref{proposition:nef-tangent}, for some $b\in B$ there is a surjection $M_{\mathcal{E}'}\vert_{C_b} \to Q$ where $Q$ is a sheaf with $\deg N_{f/Y} \geq \deg Q.$
Thus, $Q\langle \delta L'\rangle$ is nef as well, and 
\begin{align*}
    \deg N_{f/Y}\langle \delta f^* L\rangle = \deg N_{f/Y} + \delta (n-2) \deg f^* L  
    &\geq \deg Q+ \delta (n-2) \deg f^* L \\
    &= \deg Q+ \delta (n-2) \deg f'^*L = \deg Q\langle \delta L'\rangle\ge 0.
\end{align*}

Now $$\deg N_{f/Y}\langle \delta f^* L\rangle = 2g(C) - 2 - (K_X+\mathcal{E})\cdot f(C) + \delta(n-2) L \cdot f(C).$$ It follows from the assumption that $K_X+\mathcal{E}-\delta(n-2)L$ is ample that there is a small enough rational number $\varepsilon>0$ such that $K_X+\mathcal{E}-\delta(n-2)L-\varepsilon L$ is still ample.
Therefore, we have
\begin{equation*}
    2g(C)-2 \geq (K_X+\mathcal{E}-\delta(n-2)L)\cdot f(C) 
    =(K_X+\mathcal{E}-\delta(n-2)L-\varepsilon L)\cdot f(C) +\varepsilon\, L\cdot f(C)
    > \varepsilon\, L\cdot f(C).
\end{equation*}
The proof is now complete.
\end{proof}

The following is the almost homogeneous variant of the above that can be deduced analogously by replacing the role of Proposition~\ref{proposition:nef-tangent} by Proposition \ref{proposition:almost-homogeneous}.

\begin{theorem}\label{ahvariant}
Let $X$ be a Gorenstein canonical projective variety of dimension $n\geq 3$ containing a Zariski-open set $X_0$ admitting a transitive group action by an algebraic group $G$. 
    Let $\mathcal{E}$ be a globally generated line bundle and $L$ be an ample line bundle on $X$.
    Assume that there exists a rational number $\delta>0$ such that the following conditions are satisfied:
    \begin{enumerate}
    \renewcommand{\labelenumi}{\textup{(\roman{enumi})}}
        \item[(i)] $M_{\mathcal{E}}\langle \delta L\rangle$ is nef, and
        \item[(ii)] $K_X+\mathcal{E}-\delta(n-2)L$ is ample.
    \end{enumerate}
    Then $|\mathcal{E}|$ is a hyperbolic linear system outside of $(X\setminus X_0)\cup \operatorname{Sing}(X)$.
\end{theorem}

\section{Proofs of the main results}\label{proofs}

\subsection{Hyperbolicity for regular varieties} We first prove our results for varieties with $q(X)=0$. 

\begin{proof}[Proof of Theorem~\ref{main1}]
The result is known when $(X,{L})\cong (\pp^n,\mathcal{O}_{\pp^n}(1))$ (i.e., when $\tau({L})=n+1$) (see \cite{Ein88}), so we will assume that $(X,{L})\ncong (\pp^n,\mathcal{O}_{\pp^n}(1))$.

The result follows from Theorem~\ref{thm:hyperbolicity-of-N+L} with $\mathcal{E}=K_X+mL$ and $\delta=1$.
Since $L$ is ample and base-point free, $K_X+mL$ is globally generated and $M_{K_X+mL}\langle L\rangle$ is nef by Theorem~\ref{thm1} as $m\geq n+1$. 
Note that
$$K_X+\mathcal{E}-(n-2)L = 2K_X + (m-n+2)L$$
which is ample since $\frac{m-n+2}{2}>\tau(L)$.
\end{proof}


We include the following corollary to highlight the use of Theorem \ref{ahvariant} which can be proven analogously:

\begin{corollary}\label{cor-ahvariant}
    Let $X$ be a smooth regular projective variety of dimension $n\geq 3$ and let
${L}$ be a very ample line bundle on $X$.
Assume that $X$ contains a Zariski-open set $X_0$ admitting a transitive group action by an algebraic group $G$.
Then the linear system $|K_X+mL|$ is pseudo-hyperbolic for $m>\max\left\{n,n+2\tau({L})-2\right\}$.
\end{corollary}

\subsection{Generalities on hyperbolicity for varieties with nef tangent bundles} We now prove the general results for varieties $X$ with nef $\mathcal{T}_X$. Special classes of such varieties  will be dealt with in the next subsection.

\begin{proof}[Proof of Theorem \ref{main4}]
Using Theorem~\ref{thm:hyperbolicity-of-N+L} with $\delta=1$, we need to verify that
\begin{itemize}
    \item $M_{K_X+mL}(L)$ is nef, and
    \item $2K_X+(m-n+2)L$ is ample.
\end{itemize}
The second one follows from our assumption that $m>n+2\tau(L)-2$, while the first one is a consequence of Corollary \ref{betacor} (and Remark \ref{rmk-va}).
\end{proof}

\begin{proof}[Proof of Corollary \ref{cor4}] Follows immediately from Theorem \ref{main4} {since $\beta(L_A)\leq 1$}.
\end{proof}

\subsection{Hyperbolicity for abelian varieties} 
First, we will prove Theorem \ref{mainmod}. For this purpose, we establish two auxiliary results:

\begin{proposition}\label{prop_abelain_suff_condition}
Let $X$ be an abelian variety of dimension $g$ and $L$ an ample and base-point free line bundle on $X$.
Assume that for any translation $Y$ of any abelian subvariety of $X$ with $1 \leq \dim Y \leq g-2$,
the rank of the restriction $H^0(L) \to H^0(L|_Y)$ is greater than $g -\dim Y$.
Then $|L|$ is hyperbolic.
\end{proposition}

\begin{proof}
Assume that $|{L}| $ is not hyperbolic.
By Remark \ref{Bloch}, there exists an abelian subvariety  $Z \subset X$ with $\dim Z \geq 1$ such that
very general (or any) $D \in |L| $ contains a translation of $Z$.
Since a general $D$ is irreducible and  is not a translation of an abelian subvariety, $\dim Z \leq g-2$.
Let $\pi : X \to X/Z$ be the canonical quotient morphism
and let
\begin{align*}
\Sigma \coloneqq \{ (D,t) \in |{L}| \times (X/Z) \mid \pi^{-1}(t) \subset D \}.
\end{align*}
By assumption,
the natural projection $\Sigma \to |{L}|$ is surjective.
For each $t \in X/Z$, we have an exact sequence
\begin{align*}
0 \to H^0(L \otimes \mathcal{I}_{\pi^{-1}(t)}) \to H^0(L) \xrightarrow{r} H^0(L|_{\pi^{-1}(t)}).
\end{align*}
Since $\pi^{-1}(t)$ is a translation of $Z$,
$\dim \mathrm{Im} \, r > g -\dim Z$ by assumption and hence $$h^0(L \otimes \mathcal{I}_{\pi^{-1}(t)}) < h^0( L) -(g -\dim Z).$$
Since the fiber of $\Sigma \to X/Z$ over $t$ is $|H^0({L} \otimes \mathcal{I}_{\pi^{-1}(t)}) |$,
it holds that
\begin{align*}
\dim \Sigma < \dim X/Z +  h^0( L) -(g -\dim Z) -1=h^0(L) -1 = \dim |{L}|,
\end{align*}
which contradicts the surjectivity of  $\Sigma \to |{L}|$.
\end{proof}

\begin{lemma}\label{lem_jet_bundle}
Let $ H $ be an ample line bundle on an abelian variety $Y$.
Let $p \geq 0$ be an integer and $V \subset H^0(Y,H)$ be a subspace.
Assume that $$V \hookrightarrow H^0(Y,H ) \to H \otimes \mathcal{O}_Y/\mathcal{I}_y^{p+1}$$ is surjective for any $y \in Y$.
Then $\dim V > \binom{p+\dim Y}{p}$ holds.
\end{lemma}

\begin{proof}
We have $\dim V \geq  \dim H \otimes \mathcal{O}_Y/\mathcal{I}_x^{p+1}= \binom{p+\dim Y}{p}$.
Assume $\dim V= \binom{p+\dim Y}{p}$.

Let $\Delta \subset Y \times Y$ be the diagonal and let $p_i : Y \times Y \to Y$ be the projection to the $i$-th factor for $i=1,2$.
By assumption, the natural map
\[
V \otimes \mathcal{O}_Y  \hookrightarrow H^0(Y,H)  \otimes \mathcal{O}_Y  \to {p_2}_* (p_1^* H \otimes \mathcal{O}_{Y \times Y}/\mathcal{I}_{\Delta}^{p+1}) \eqqcolon  J^{p}(H)
\]
is a surjection between vector bundles of the same rank, and hence is an isomorphism.
On the other hand, the degree of the jet bundle $J^{p}(H)$ is $$ \deg H  \cdot \rank  J^{p}(H) >0$$ by $\Omega^1_Y \simeq \mathcal{O}_Y^{\oplus \dim Y}$, which is a contradiction.
Hence $\dim V>  \binom{p+\dim Y}{p}$ holds.
\end{proof}

We are now ready to provide the 

\begin{proof}[Proof of Theorem~\ref{mainmod}]
Let $Y \subset X$ be a translation of an abelian subvariety of dimension $d \geq k+1$.
Let $V$  be the image of the restriction $H^0(L) \to H^0(L|_Y)$.
Since $L$ separates $p$-jets at any $x\in X$,
the map $$V \to L|_Y \otimes \mathcal{O}_Y/\mathcal{I}_y^{p+1}$$ is surjective for any $y \in Y$.
To prove the Kobayashi hyperbolicity of $|L|$,
it suffices to show $\dim V > g-d$ by Proposition \ref{prop_abelain_suff_condition}.
By Lemma \ref{lem_jet_bundle},
$\dim V > g-d$ holds if
$$\binom{p+d}{p}  \geq g-d.$$
Since $\binom{p+ d}{p} +d $ is increasing in $d$,
\eqref{eq_cor_p-jet_ample} implies $\binom{p+d}{p} +d \geq g$ for any $d \geq k+1$.

The last statement follows from Theorem \ref{It}. 
\end{proof}

\begin{remark}\label{rem_proof_of_first_version} \em{The last statement of Remark \ref{rmk1} says that } 
\begin{center}
    {\it the linear system $|L|$ is hyperbolic if 
$\mathcal{I}_o \langle \frac{1}{g-1} L \rangle $ is M-regular}
\end{center}
for any ample line bundle $L$ on an abelian variety $X$ of dimension $g \geq 3$.
An alternate way of seeing this using Theorem~\ref{thm:hyperbolicity-of-N+L} is as follows. Since $\mathcal{I}_o\langle \frac{1}{g-1}L \rangle$ is M-regular and $g\geq 3$,
$\mathcal{I}_o\langle \frac12L \rangle$ is M-regular as well.
Hence $L$ is very ample by \cite[Theorem 1.6]{MR4474904}.
Furthermore,
$M_L\langle \frac{1}{g-2} L \rangle$ is M-regular by Proposition \ref{prop_JP_8.1} and hence ample.
Then there is a rational number $0<\delta<\frac1{g-2}$ such that $M_{L}\langle\delta L\rangle$ is still ample.
To finish, we can apply Theorem~\ref{thm:hyperbolicity-of-N+L} to $\mathcal{E}=L$ and this $\delta$
since $$K_X+\mathcal{E}-\delta(g-2)L = (1-\delta(g-2))L $$ is ample.
\end{remark}

\begin{proof}[Proof of Corollary \ref{cor_KH_mL}]
    This follows from the last statement of Theorem \ref{mainmod}.
\end{proof}

\begin{proof}[Proof of Corollary \ref{cor_g-k}]
    If $k\geq g-2$, then any general $D \in |L|$ does not contain abelian subvarieties of positive dimension and hence is hyperbolic.
Hence we may assume $0 \leq k \leq g-3$.

Since  $\binom{m+k-1}{k+1} + k+1 $ and $\binom{m+k}{k+1} + k+1 $ are increasing in $m$,
it suffices to show that $$\binom{g-1}{k+1} + k+1 -g \geq 0$$ by Corollary \ref{cor_KH_mL}.
We fix $k$ and set $$f(g) = \binom{g-1}{k+1} + k+1 -g.$$ 
If $g=k+3$, we have $f(k+3) =k \geq 0$.
Since $f(g+1) -f(g) = \binom{g-1}{k} -1 \geq 0$ for $g \geq k+3$,
$f(g) \geq 0$ holds for any $g \geq k+3$.
\end{proof}

\begin{proof}[Proof of Corollary~\ref{cor-ab-2}]
For the line bundle ${L}$, we define the following invariant
\[
r(L):=
{\rm inf} \{ 
c\in \qq \mid \text{
there exists an effective $\qq$-divisor
$D\equiv cL$ such that 
$\mathcal{J}(X,D)=\mathcal{I}_o$
}
\}
\]
where $\mathcal{I}_o$ denotes the ideal sheaf of the origin in the abelian variety
and $\mathcal{J}(X,D)$ denotes the multiplier ideal of the pair.
If the inequality 
\[
d_1 \cdots d_g > \frac{4^g(g-1)^gg^g}{2g!} 
\]
holds and $(X, L)$ is very general, then the Seshadri constant
$\epsilon(X,L)$ is larger than $(g-1)g$ (see~\cite[Theorem 1.(b)]{Bau98}).
By~\cite[Lemma 1.2]{LPP}, we know that $r(L) < \frac{1}{g-1}$ provided that the Seshadri constant
$\epsilon(X,L)$ is larger than $g(g-1)$. 
On the other hand, by~\cite[Proposition 1.4]{Cau}, we know that $\beta(\underline{l})\leq r(L)$. 
Putting these inequalities together, we conclude that $\beta(L) < \frac{1}{g-1}$ if $(X, L)$ is very general. 
Then $\beta(L) < \frac{1}{g-1}$ holds if $(X, L)$ is general by \cite[Theorem 3.1]{Ito21}
and hence Theorem~\ref{mainmod} finishes the proof.
\end{proof}

\begin{proof}[Proof of Corollary~\ref{cor-ab-3}]
Follows from Theorem~\ref{mainmod} and~\cite[Theorem 1.5.(2)]{Ito21}.
\end{proof}

In order to prove Proposition \ref{intro-prop1-Ab}, we first need the following lemma.

\begin{lemma}\label{lem:alg-hyper-product-Ab}
Let $(X_i, L_i) $ be polarized abelian varieties for $i=1,2$ of positive dimensions,
and set $$(X, L) = (X_1 \times X_2,   L_1 \boxtimes  L_2).$$
If $\dim X_1 \geq h^0(X_2, L_2)$,
then any section of $ | L| $ contains a fiber of the natural projection $X_1 \times X_2 \to X_1$.
\end{lemma}

\begin{proof} 
Take $s \in H^0(X, L) =H^0(X_1, L_1) \otimes H^0(X_2, L_2) $ and 
let $D \subset X$ be the divisor defined by $s$.
For $x_1 \in X_1$,
\begin{align*}
D \supset \{x_1\} \times X_2 &\Leftrightarrow s|_{\{x_1\} \times X_2} =0 \in ( L_1 \otimes k(x_1)) \otimes H^0(X_2, L_2)\\
& \Leftrightarrow  x_1 \in \mathrm{Zero}({p_1}_* s ),
\end{align*}
where $\mathrm{Zero}({p_1}_* s ) \subset X_1$ is the zero locus of the induced $ {p_1}_* s  : \mathcal{O}_{X_1} \to  L_1  \otimes H^0(X_2, L_2)$ on $X_1$.
Thus if $\dim X_1 \geq h^0(X_2, L_2)$,
$\mathrm{Zero}({p_1}_* s ) \neq \emptyset$ since $ L_1$ is ample.
Hence $D$ contains the fiber $\{x_1\} \times X_2$ over $x_1 \in \mathrm{Zero}({p_1}_* s ) $.
\end{proof}

\begin{proof}[Proof of Proposition \ref{intro-prop1-Ab}]
Let $(X,L)=(X',L') \times (E,\Theta)$ be as in the statement.
Then any section of $|(g-1) L|$ contains an elliptic curve, due to Lemma~\ref{lem:alg-hyper-product-Ab}, 
since $$ \dim X'=g-1 =h^0((g-1) \Theta).$$
Hence the linear system $|(g-1) L|$ is not hyperbolic by definition.
\end{proof}

We can give an alternative proof of \cite[Theorem A]{caucci_KH} using jet-separation:

\begin{theorem}[{\cite[Theorem A]{caucci_KH}}]\label{caa}
Let $X$ be an abelian variety of dimension $g \geq 3$ and  $L$ be an ample line bundle on $X$.
Assume that $$(X,L)\ncong(X',L') \times (E,\Theta)$$ with $ (E,\Theta)$ a principally polarized elliptic curve.
Then the linear system $|(g-1)L|$ is hyperbolic.
\end{theorem}

\begin{proof}
By Corollary \ref{cor_g-k}(2) for $k=0$,
we may assume that $L$ has base divisors.
Then $(X,L)$ is isomorphic to 
\[
(X',L') \times (A_1, \Theta_1) \times \cdots \times (A_l, \Theta_l)
\]
where $L'$ has no base divisors and each $(A_i, \Theta_i)$ is an indecomposable  p.p.a.v.\ with $\dim A_i \geq 2$.
For a base-point free line bundle $N$,  $f_{N}$ denotes the morphism defined by the complete linear system $|N|$.
By K\"unneth formula,
the morphism $ f_{2L}$ is the product $f_{2L'} \times f_{2\Theta_1} \times \cdots \times f_{2\Theta_l}$.
By \cite{MR871633}, $f_{2L'} $ is very ample.
On the other hand,
$f_{2\Theta_i}  $ is the natural morphism $A_i \to A_i/\langle -1_{A_i} \rangle$ to the Kummer variety up to a translation.
In particular, $f_{2\Theta_i}  $ is an immersion outside the set of  two-torsion points,  whose codimension is $\dim A_i \geq 2$.
Hence $ f_{2L}$ is an immersion outside a subset whose codimension is at least two.
Thus $ \mathcal{I}_o^2 \otimes 2L$  is M-regular by \cite[Example 5.4 (1)]{MR4474904}.
Then $ \mathcal{I}_o^h \otimes hL$ is IT$_0$ for $h \geq 3$ by \cite[Proposition 3.1 (ii)]{MR4474904}.

If $g \geq 4$, then $ \mathcal{I}_o^{g-1} \otimes (g-1)L$ is IT$_0$ and hence 
$h^1( \mathcal{I}_x^{g-1} \otimes (g-1)L)=0$ for any $x \in X$.
This means that $(g-1)L$ separates $(g-2)$-jets at any $x $ and hence $|(g-1)L|$ is hyperbolic by Remark \ref{rmk1}.
  
Assume $g=3$.
Then $(X,L) =(A,\Theta)$ is an indecomposable p.p.a.v., or 
$(X,L) =(X',L') \times (A, \Theta)$ with $\dim X'=1, \deg L' \geq 2 $ and indecomposable p.p.a.v. $(A,\Theta)$.
We prove the hyperbolicity of $|2L|$ by modifying the proof of Proposition \ref{prop_abelain_suff_condition} as follows.

Let $C \subset X$ be an elliptic curve and let $r_C : H^0(2L) \to H^0(2L|_C)$ be the restriction map.
Since $ 2L$ is base-point free and ample, we have $\rank r_C \geq 2$.
Assume $\rank r_C = 2$.
Then the image $f_{2L} (C) $ is isomorphic to a line in $ \mathbb{P}(H^0(2L))$.

Let $Z \subset X$ be an abelian subvariety of $\dim Z=1$.
Let $\pi : X \to X/Z$ be the canonical quotient morphism
and let
\begin{align*}
\Sigma \coloneqq \{ (D,t) \in |{2L}| \times (X/Z) \mid \pi^{-1}(t) \subset D \}.
\end{align*}
To prove the hyperbolicity of $|2L|$,
it suffices to show $ \dim \Sigma < \dim |2L|$.

Let 
\[
W \coloneqq \{ t \in X/Z \mid \rank r_{\pi^{-1}(t)}=2\}=\{ t \in X/Z \mid  f_{2L}(\pi^{-1}(t)) \text{ is a line}\}.
\]
If $f_{2L}(\pi^{-1}(t)) $ is a line, 
then $\pi^{-1}(t) \to f_{2L}(\pi^{-1}(t))$ is not \'etale and hence
$f_{2L}$ is not an immersion at some point in $\pi^{-1}(t)$.
Hence $\pi^{-1}(t) \cap A_{[2]} \neq \emptyset$ if $(X,L) =(A,\Theta)$,
and $\pi^{-1}(t) \cap (X' \times A_{[2]}) \neq \emptyset$ if $(X,L) =(X',L') \times (A,\Theta)$,
where $A_{[2]}$ is the set of two torsion points of $A$.
Since $A_{[2]}$ is a finite set, $\dim X' =1$ and $\dim X/Z=2$,
we see that $W$ is a proper closed subset of $X/Z$.

Since 
$\rank r_{\pi^{-1}(t) } \geq 3$ for $t \in (X/Z) \setminus W$,
we have
\begin{align*}
\dim \Sigma \leq  \max \{ \dim (X/Z) \setminus W + h^0(2L) - 4  , \dim W + h^0(2L) - 3 \} =h^0(2L) -2 < \dim |2L|.
\end{align*}
Thus the projection $\Sigma \to |2L| $ is not surjective and hence $|2L|$ is hyperbolic. 
\end{proof}

\subsection{Hyperbolicity for homogeneous, hyperelliptic and Kummer varieties} We end this article by proving the last three results mentioned in the introduction.

\begin{proof}[Proof of Theorem \ref{thm3'}]
Apply the Borel--Remmert structure theorem to see that $X=A\times Y$ where $A$ is abelian and $Y$ is rational homogeneous. Note that $\dim Y=n-q(X)$ and $\dim A=q(X)$. Furthermore, $L=L_1\boxtimes L_2$ where $L_1$ (resp. $L_2)$ is an ample line bundle on $A$ (resp. $Y$). Also note that $L_2$ is very ample.

(1) Note that $m>n-1\geq n-q(X)$. By Theorem~\ref{thm:hyperbolicity-of-N+L} and Proposition \ref{prop2'}, we need to look for a positive rational number $\delta$ such that 
\begin{itemize}
    \item $M_{mL_1}\langle \delta L_1\rangle$ is nef,
    \item $M_{K_Y+mL_2}\langle \delta L_2\rangle$ is nef, and 
    \item $2K_X+mL-\delta(n-2)L$ is ample, which is satisfied if $\frac{m-\delta(n-2)}{2}>\tau(L)$.
\end{itemize}
Since $m\geq 2$ by assumption, by Remark \ref{ab-g+1'}(2), the first condition is satisfied if $\delta\geq \frac{m}{m-1}$. The second condition is also satisfied by any $\delta$ in this range by Theorem \ref{thm1}, 
except when $Y\simeq \mathbb{P}^{n-q(X)}$, $L_2\simeq \mathcal{O}_{\mathbb{P}^{n-q(X)}}(1)$, and $m=n-q(X)+1$ hold simultaneously.
In this particular case, condition (i) in Theorem~\ref{thm:hyperbolicity-of-N+L} with $\delta\geq\frac{m}{m-1}$ is met anyway, since $K_Y+mL_2$ is trivial, and $M_{mL_1\boxtimes \mathcal{O}_Y}\langle\delta(L_1\boxtimes L_2)\rangle \cong (M_{mL_1}\boxtimes \mathcal{O}_Y)\langle\delta(L_1\boxtimes L_2)\rangle$ is nef. Now the third condition is equivalent to $\delta<\frac{m-2\tau(L)}{n-2}$. Thus, to prove the existence of a rational number $\delta$ satisfying the three conditions, it is enough to show that $$\frac{m}{m-1}<\frac{m-2\tau(L)}{n-2},$$ which upon simplification gives $$m^2-m(n+2\tau(L)-1)+2\tau(L)>0.$$
It is easy to check that the above is satisfied when $m\geq n+2\tau(L)-1$.

(2) The proof is analogous but here we take $\delta=1$. The last two conditions (described in the proof of the first part) can be verified analogously. To verify the first condition, we can use Corollary \ref{ab-g+1} instead of Remark \ref{ab-g+1'}(2) {since $m > n+2\tau(L) -2 \geq q(X)-1=\dim A-1$}.
\end{proof}

\begin{proof}[{Proof of Theorem \ref{hyp}}]
We only prove (2). Applying Theorem~\ref{thm:hyperbolicity-of-N+L} with $\delta=2$, we only need to show that 
\begin{itemize}
    \item $M_{N+mL}(2L)$ is nef, and
    \item $K_X+N+mL-2(n-2)L$ is ample. 
\end{itemize}
The first one follows from Theorem \ref{caucci}(2) and the second one is obvious as $K_X\equiv 0$. 

To prove (1), just use Theorem \ref{caucci}(1) instead.
\end{proof}

As explained in \cite{Cau25}, there are many varieties that satisfy the hypothesis of Theorem \ref{hyp}, we include one corollary along that line (see {\it loc. cit.} for details):

\begin{corollary}\label{BdF}
    Let $X$ be a  Bagnera-de Franchis variety of dimension $n\geq 3$. Then, for any ample (resp. numerically trivial) line bundle $L$ (resp. $N$)  on $X$, the linear system $|N+mL|$ is hyperbolic for $m\geq 2n$.
\end{corollary}

\begin{proof}[Proof of Theorem \ref{thm-kummer}]
In the proof of Theorem~\ref{thm:hyperbolicity-of-N+L},
we only use the positivity of $\mathcal{T}_X$ to argue that 
$\deg (Q'|_C) \geq 0$ for a quotient coherent sheaf  $\mathcal{T}_X\rightarrow Q'\rightarrow 0$ and curves 
$C\subset Y$ contained in general elements of $Y\in |\mathcal{E}|$.  
In the case of a Kummer variety, $K(A)$ has isolated singularities. 
For any $m \geq 1$,
the line bundle $\mathcal{E}:=mL$ is base-point free by~\cite[Lemma 2.5, Page 12]{MR4688553}. Therefore, a general element $Y\in |mL|$ is contained in the smooth locus of $K(A)$. By Lemma~\ref{kummer}, we know that 
$\deg(Q'|_C)\geq 0$ for any quotient $Q'$ of $\mathcal{T}_X$ and any curve $C\subset K(A)$ which avoids the singular locus. In particular, such property holds for any $C\subset Y$ with $Y\in |mL|$ general.
 
On the other hand, it follows from Lemma \ref{pos-kummer} and Lemma \ref{gvn} that
$\pi^* M_{mL} \langle \frac{m}{2(m-1)} \pi^*L \rangle$ is ample and hence so is 
$M_{mL} \langle \frac{m}{2(m-1)} L\rangle$.
Consequently $M_{mL} \langle \delta L \rangle$ is nef if  $\delta =\frac{m}{2(m-1)} -\varepsilon  $ for $0 < \varepsilon \ll 1$.
For $n=\dim K(A) \geq 3$, we have
\[
K_{K(A)} +mL -\delta(n-2) L \equiv \left(m -\frac{m(n-2)}{2(m-1)}+ (n-2)\varepsilon \right)L
\]
is ample if and only if $m \geq \frac{n}2$. 
Therefore, the proof of Theorem~\ref{thm:hyperbolicity-of-N+L}
applies in this setting to imply that $|mL|$ is hyperbolic.
\end{proof}
\bibliographystyle{habbvr}
\bibliography{bib}

\end{document}